\documentclass[preprint]{elsarticle}

\usepackage{lineno,hyperref}
\usepackage{amssymb}
\usepackage{amsmath}
\usepackage{bbm}

\journal{Journal of Functional Analysis}

\usepackage{tikz-cd}

\begin{document}

\newtheorem{theorem}{Theorem}
\newtheorem{acknowledgment}[theorem]{Acknowledgment}
\newtheorem{corollary}[theorem]{Corollary}
\newtheorem{definition}[theorem]{Definition}
\newtheorem{assumption}[theorem]{ Complementary Assumptions }
\newtheorem{example}[theorem]{Example}
\newtheorem{lemma}[theorem]{Lemma}
\newtheorem{notation}[theorem]{Notation}
\newtheorem{problem}[theorem]{Problem}
\newtheorem{proposition}[theorem]{Proposition}
\newtheorem{question}[theorem]{Question}
\newtheorem{remark}[theorem]{Remark}
\newtheorem{setting}[theorem]{Setting}
\newtheorem{assertion}[theorem]{Assertion}
\newproof{proof}{Proof}

\numberwithin{theorem}{section}
\numberwithin{equation}{section}


\newcommand{\1}{{\bf 1}}
\newcommand{\Ad}{{\rm Ad}}
\newcommand{\Aut}{{\rm Aut}\,}
\newcommand{\ad}{{\rm ad}}
\newcommand{\botimes}{\bar{\otimes}}
\newcommand{\Ci}{{\mathcal C}^\infty}
\newcommand{\Der}{{\rm Der}\,}
\newcommand{\de}{{\rm d}}
\newcommand{\ee}{{\rm e}}
\newcommand{\End}{{\rm End}\,}
\newcommand{\ev}{{\rm ev}}
\newcommand{\GL}{{\rm GL}}
\newcommand{\Gr}{{\rm Gr}}
\newcommand{\hotimes}{\widehat{\otimes}}
\newcommand{\id}{{\rm id}}
\newcommand{\ie}{{\rm i}}
\newcommand{\gl}{{{\mathfrak g}{\mathfrak l}}}
\newcommand{\Ker}{\ker}
\newcommand{\Lie}{\text{\bf L}}
\newcommand{\pr}{{\rm pr}}
\newcommand{\Ran}{{\rm Ran}\,}
\renewcommand{\Re}{{\rm Re}\,}
\newcommand{\spa}{{\rm span}\,}
\newcommand{\Tr}{{\rm Tr}\,}
\newcommand{\U}{{\rm U}}

\newcommand{\CC}{{\mathbb C}}
\newcommand{\HH}{{\mathbb H}}
\newcommand{\RR}{{\mathbb R}}
\newcommand{\TT}{{\mathbb T}}

\newcommand{\Ac}{{\mathcal A}}
\newcommand{\Bc}{{\mathcal B}}
\newcommand{\Cc}{{\mathcal C}}
\newcommand{\Dc}{{\mathcal D}}
\newcommand{\Ec}{{\mathcal E}}
\newcommand{\Fc}{{\mathcal F}}
\newcommand{\Hc}{{\mathcal H}}
\newcommand{\Jc}{{\mathcal J}}
\newcommand{\Kc}{{\mathcal K}}
\newcommand{\Mc}{{\mathcal M}}
\newcommand{\Nc}{{\mathcal N}}
\newcommand{\Oc}{{\mathcal O}}
\newcommand{\Pc}{{\mathcal P}}
\newcommand{\Sc}{{\mathcal S}}
\newcommand{\Tc}{{\mathcal T}}
\newcommand{\Vc}{{\mathcal V}}
\newcommand{\Uc}{{\mathcal U}}
\newcommand{\Xc}{{\mathcal X}}
\newcommand{\Yc}{{\mathcal Y}}
\newcommand{\Gc}{{\mathcal G}}
\newcommand{\Rc}{{\mathcal R}}

\newcommand{\Ag}{{\mathfrak A}}
\newcommand{\Bg}{{\mathfrak B}}
\newcommand{\Fg}{{\mathfrak F}}
\newcommand{\Gg}{{\mathfrak G}}
\newcommand{\Ig}{{\mathfrak I}}
\newcommand{\Jg}{{\mathfrak J}}
\newcommand{\Lg}{{\mathfrak L}}
\newcommand{\Mg}{{\mathfrak M}}
\newcommand{\Pg}{{\mathfrak P}}
\newcommand{\Sg}{{\mathfrak S}}
\newcommand{\Xg}{{\mathfrak X}}
\newcommand{\Yg}{{\mathfrak Y}}
\newcommand{\Zg}{{\mathfrak Z}}

\newcommand{\ag}{{\mathfrak a}}
\newcommand{\bg}{{\mathfrak b}}
\newcommand{\dg}{{\mathfrak d}}
\renewcommand{\gg}{{\mathfrak g}}
\newcommand{\hg}{{\mathfrak h}}
\newcommand{\kg}{{\mathfrak k}}
\newcommand{\mg}{{\mathfrak m}}
\newcommand{\n}{{\mathfrak n}}
\newcommand{\og}{{\mathfrak o}}
\newcommand{\pg}{{\mathfrak p}}
\newcommand{\sg}{{\mathfrak s}}
\newcommand{\tg}{{\mathfrak t}}
\newcommand{\ug}{{\mathfrak u}}
\newcommand{\zg}{{\mathfrak z}}

\newcommand{\ZZ}{\mathbb Z}
\newcommand{\NN}{\mathbb N}
\newcommand{\BB}{\mathbb B}
\newcommand{\FF}{\mathbb F}

\newcommand{\ap}{\rightarrow}
\newcommand{\tto}{\rightrightarrows}

\def\dis{\displaystyle}

\def\ap{\rightarrow}
 \def\a{\alpha}
\def\b{\beta}
\def\g{\gamma}
\def\G{\Gamma}
\def\D{\triangle}
\def\d{\delta}
\def\o{\omega}
\def\th{\theta}
\def\L{\Lambda}
\def\n{\nu}
\def\di{\diamond}
\def\r{\rho}
\def\ps{\psi}
\def\m{\mu}
\def\n{\nu}
\def\s{\sigma}
\def\sou{\overline}
\def\so{\underline}
\def\na{\nabla}
\def\O{\Omega}
\def\T{\mathbf{T}}
\def\V{\mathbf{V}}
\def\H{\mathbf{H}}


\begin{frontmatter}

\title{
	Banach-Lie groupoids and generalized inversion
}

\author{Daniel Belti\c t\u a}
\ead{beltita@gmail.com}
\address{\small Institute of Mathematics ``S. Stoilow'' of the Romanian Academy, 21 Calea Grivitei Street, 010702 Bucharest, Romania}

\author{Tomasz Goli\'nski}
\ead{tomaszg@math.uwb.edu.pl}

\author{Grzegorz Jakimowicz}
\ead{g.jakimowicz@uwb.edu.pl}
\address{\small University in Bia{\l}ystok, Institute of Mathematics, Cio{\l}kowskiego 1M, 15-245~Bia{\l}ystok, Poland}

\author{Fernand Pelletier}
\ead{fernand.pelletier@univ-smb.fr}
\address{\small Unit\'e Mixte de Recherche 5127 CNRS, Universit\'e  de Savoie Mont Blanc, Laboratoire de Math\'ematiques (LAMA),Campus Scientifique,  73370 Le Bourget-du-Lac, France}

\begin{abstract} 
We study a few basic properties of Banach-Lie groupoids and algebroids,  
adapting some  classical results on  finite dimensional Lie groupoids. 
As an illustration of the general theory, 
we show that the notion of locally transitive Banach-Lie groupoid sheds fresh light on earlier research on some 
infinite-dimensional manifolds associated with Banach algebras.
\end{abstract}


\end{frontmatter}


\setcounter{tocdepth}{1}
\tableofcontents

\section{Introduction}

The theory of Banach-Lie groups (i.e., infinite-dimensional Lie groups modeled on Banach spaces) is  a rather old and well-developed research area that is interesting on its own and also for its applications to many problems in functional analysis, differential geometry, or mathematical physics. 
There exist however several topics whose natural background requires Banach manifolds endowed with an algebraic structure that is more general than the notion of group. 
We will briefly mention here very few references on some of these topics that 
can be better understood from the perspective of Banach-Lie groupoids, which is the main theme of our present paper: 
\begin{itemize}
	\item Moore-Penrose pseudo-inverses in $C^*$-algebras 
	and their differentiability properties  
	(cf. 
	\cite{ACM05}, \cite{Boa06}, \cite{ArCG08}, \cite{LeRo12}, \cite{ArM13}); 
	\item Poisson structures on the predual of a $W^*$-algebra 
	(cf. \cite{OJS15} and \cite{OS11}); 
	\item Banach-Lie algebroids (cf. \cite{An11}, \cite{CaPe}, \cite{Pe}).
\end{itemize}

Motivated by the above research directions, and also by the impressive development of the theory of finite-dimensional Lie groupoids (see for instance \cite{Ma}, \cite{MoMr}, \cite{CraFe}) we think it worthwhile to develop the basic theory of Banach-Lie groupoids and to illustrate it by a brief discussion on its relation to some differentiability questions in the theory of $C^*$-algebras. 
More specifically, the contents of this paper are as follows. 
Section~\ref{Sect2} collects some basic notions on differential geometry and on topological groupoids that we need. In Section~\ref{Sect3} we introduce the notion of (not-necessarily-Hausdorff) Banach-Lie groupoid 
along with several examples. 
The main result here (Theorem~\ref{5.4}) concerns the differentiability properties of the orbits, extending the classical results on actions of Banach-Lie groups. 
In Section~\ref{Sect4} we extend the Lie functor from Banach-Lie groups to Banach-Lie groupoids (Theorem~\ref{algbroidgroupoid}) and we establish the link between the orbits of a split Banach-Lie groupoid and the orbits of its corresponding Banach-Lie algebroid (Theorem~\ref{linkorbitG}). 
In Section~\ref{Sect5} we obtain a Banach-Lie groupoid version of the fact that for every Banach-Lie group there exists a simply connected Banach-Lie group with the same Lie algebra (Theorem~\ref{ssymply}). 
Section~\ref{Sect6} is devoted to the study of locally transitive Banach-Lie groupoids, a class of groupoids that play a central role in the description of the differentiability properties of the Moore-Penrose pseudo-inverse in $C^*$-algebras. 
Among other things, we provide several equivalent characterizations of the groupoids of this type (Theorem~\ref{transitive}) 
and we study the Atiyah bundles associated to Banach principal bundles (Proposition~\ref{EG}). 
Finally, in Section~\ref{Sect7} we briefly illustrate the general theory developed so far,  
by the locally transitive Banach-Lie groupoids associated to unital Banach algebras. 
As mentioned above, this provides a natural framework for some results that were established in the earlier literature in connection with the Moore-Penrose pseudo-inverse.

\section{Preliminaries and notation}\label{Sect2}

\subsection{The context of not-necessarily-Hausdorff Banach manifolds}\label{subsect-nnH}
The terminology in the literature on infinite-dimensional manifolds is not uniform, so we mention that 
our general references are \cite[Ch. II--III]{La} or \cite{Bo09}. 
For the sake of clarity and for later reference in this paper, we briefly recall here some classical notions on not-necessarily-Hausdorff manifolds modeled on Banach spaces.

A $C^\infty$-{\it atlas}   on a set $M$   is a family $\{(U_\alpha, u_\alpha) \}_{\alpha\in A}$ of subsets $U_\alpha$ of $M$ 
and maps~$u_\alpha$ from $U_\alpha$ to a Banach space $\mathbb{M}_\a$ such that:
\begin{itemize}
\item  $u_\alpha$ is a bijection of $U_\alpha$ onto a open  subset of   $\mathbb{M}_\alpha$ for all $\alpha\in A$;

\item $M=\displaystyle\bigcup_{\alpha\in A}U_\alpha$;

\item for any $\alpha$ and $\beta$ such that  $U_{\alpha\beta}=U_\alpha\cap U_\beta\not=\emptyset$, 
then $u_{\alpha\beta}=u_\alpha\circ u_\beta^{-1}:u_\beta(U_{\alpha\beta})\subset \mathbb{M}_\beta\longrightarrow u_\alpha(U_{\alpha\beta})\subset \mathbb{M}_\alpha$ is a smooth map.
\end{itemize}
As usual, one has the notion of equivalent  $C^\infty$-atlases on  $M$. 
An equivalence class of $C^\infty$-atlases on $M$ 
defines a topology on $M$ which in general fails to have the Hausdorff property.

\begin{definition}\label{Banach} 
\normalfont
An equivalence class of $C^\infty$-atlases
is called a \emph{not-necessarily-Hausdorff Banach manifold structure} on $M$, 
for short a \emph{n.n.H. Banach manifold}. 
This structure is called a \emph{Hausdorff Banach manifold  structure} on $M$ 
(for short a \emph{Banach manifold}  as in \cite{Bo09}) if the topology defined by this atlas is a Hausdorff topology.
\end{definition}

It follows by the above definition that all Banach spaces $\mathbb{M}_\a$ are topologically isomorphic on every connected component of~$M$. 
If all connected  components of $M$ are modelled on a fixed Banach space $\mathbb{M}$ (up to an isomorphism)  
then we will say that $M$ is a \emph{pure Banach manifold}.  
A  \emph{pure Banach component} of   a   Banach manifold $M$ is  a pure  Banach manifold  which is a union
$${N}=\bigcup_{\alpha\in A}M_\alpha$$ 
of connected components of $M$. 
We say that $N$ is a \emph{maximal pure (n.n.H.) Banach component} if for any connected components  $M_\lambda$ of $M$ such that $M_\lambda\cap N=\emptyset$ then $M_\lambda$ is modelled on a Banach space $\mathbb{M}_\lambda$ which is not isomorphic to the model space of~$N$.

\begin{remark}\label{nnH}  
\normalfont 
If $M$ is a n.n.H. Banach manifold, then its corresponding topology is~$T_1$ hence  each finite subset of $M$ is closed. 
Moreover, since $M$ is locally Hausdorff, by Zorn's Lemma,  there exists a maximal open dense subset $M_0$ of $M$ which is an open Banach manifold in $M$.  
(See for instance \cite[Lemma 4.2]{BaGa}.)
\end{remark}

It is clear that the classical construction of the tangent bundle $TM$ of a Banach manifold $M$ 
can be applied to a n.n.H. Banach manifold and so we get again a n.n.H. Banach manifold $TM$ which is a Banach manifold 
if and only if $M$ is. 



A smooth map between n.n.H. Banach manifolds $\varphi\colon N\to M$ is called a  \emph{weak immersion} if its tangent map $T_x\varphi:T_xN\ap T_{\varphi(x)}M$ is injective 
for every $x\in N$. 
If $\varphi$ is a weak immersion for which the range of its tangent map $T_x\varphi$ is a closed subspace of $T_{\varphi(x)}N$ for every $x\in N$, then $\varphi$ is called an \emph{immersion}. 
If $\varphi$ is an immersion for which  the range of $T_x\varphi$ is a split subspace of $T_{\varphi(x)}M$ 
(that is, there exists a closed linear subspace $\Vc$ for which one has the direct sum decomposition $T_{\varphi(x)}M=(T_x\varphi)(T_xN)\oplus\Vc$), 
then $\varphi$ is called a \emph{split immersion}. 
We emphasize that this terminology is not generally used in the literature, since for instance the split immersions in the above sense are called immersions in \cite[5.7.1]{Bo09}. 
See \cite{MaOu92} and \cite{Gl15} for additional information.

An  \emph{immersed} (resp., weakly or split) \emph{n.n.H. Banach submanifold} 
of  a n.n.H. Banach manifold $M$ is an injective immersion (resp., weak or split immersion) $\iota:{N}\ap M$. 
An  immersed n.n.H. Banach manifold 
$\iota:{N}\ap M$
is called a \emph{closed submanifold } (resp. \emph{split submanifold}) if $\iota(N)$ is a closed subset of $M$ 
(resp. $\iota$ is a split immersion).

A \emph{closed split submanifold} of a n.n.H. Banach manifold $M$ will called simply a \emph{submanifold} of $M$ and then the corresponding split immersion $\iota\colon N\to M$ is usually thought of as an inclusion map $N\hookrightarrow M$.

A  \emph{submersion} $p:N\ap M$ between two n.n.H. Banach manifolds is a surjective smooth map such that $Tp(T_xN)=T_{p(x)}M$ and $\ker T_xp$ is 
a split subspace of 
$T_xN$ for each $x\in N$. 
A  smooth map $f :N\ap M$ between two n.n.H. Banach manifolds is a \emph{subimmersion} 
if  for each $x\in N$ there exist an open neighborhood $ U $ of $x$, a Banach manifold $P$, 
a submersion $s :U \ap P$, and an immersion $ j : P \ap M$ such that $f\vert_{ U} = j \circ s$. 
If $f$ is a subimmersion, then $f^{-1}(y)$ is a submanifold of $N$ for each $y\in M$.



By locally trivial fibration  we mean a  submersion ${\bf p} \colon N\ap M$ such that for every $x\in M$ there exist an open neighborhood $U$ of $x$ and a diffeomorphism $\Phi\colon{\bf p}^{-1}(U)\ap U\times {\bf p}^{-1}(x)$ such that $p_1\circ\Phi={\bf p}$, where $p_1=U\times {\bf p}^{-1}(x)\ap U$ is the canonical projection. 
If the basis $M$ is connected then all the fibers are diffeomorphic. 

A {\it  n.n.H Banach  bundle} is a locally trivial fibration  $\pi : \mathcal{A}\ap M$ whose  fiber is a Banach space. 
Of course $\mathcal{A}$ is Hausdorff if and only if $M$ is so. 
On each connected component  $\mathcal{A}_\alpha$ of $\mathcal{A}$ the fibers are isomorphic to a common Banach space (called {\it the typical fiber}) but this space can change from one connected component to another. 
The (n.n.H) Banach bundle $\pi : \mathcal{A}\ap M$ is called {\it pure} if its fibers  have the same typical fiber and if $M$ is a pure Banach manifold. 
In particular, if $\pi:\mathcal{A}\ap M$ is a Banach bundle, each connected component  $\mathcal{A}_\alpha$ of $\mathcal{A}$ gives rise to a pure (n.n.H) Banach bundle $\pi : \mathcal{A}_\alpha\ap \pi(\mathcal{A}_\alpha)$.  For example the tangent bundle $TM$ of a (n.n.H) Banach manifold $M$ is a Banach bundle  and the tangent bundle $TM_\alpha\ap M_\alpha$ is a pure Banach bundle for each connected component  $M_\alpha$ of $M$.

Notions like n.n.H. Banach bundle morphisms, n.n.H. Banach bundle isomorphisms, n.n.H.Banach subbundles etc.,  are  defined as usual.

The algebra of smooth maps $f:M\ap \RR$ will be denoted $C^\infty(M)$, the $C^\infty(M)$-module of smooth sections of a bundle 
$\pi : \mathcal{A}\ap M$ will be denoted $\G(\mathcal{A})$, and the $C^\infty(M)$-module of smooth vector fields on $M$ will be denoted $\Xi(M)$.

\subsection{Topological groupoids}
A \emph{topological groupoid} $\mathcal{G}\tto M$ is a pair $(\mathcal{G},{M})$   of topological spaces such that  $\mathcal{G}$ may not be Hausdorff but $M$ is Hausdorff,  
with the following structure maps: 
\begin{itemize}
\item[(G1)]  
Two   surjective open continuous   maps
${\bf s}:  \mathcal{G}\ap M$ and ${\bf t}: \mathcal{G}\ap M$ 
called \emph{source} and \emph{target} maps, respectively.\footnote{ Each point $g\in G$ can be regarded as an arrow 
$g : {\bf s}(g)\ap{\bf t}(g)$ which joins ${\bf s}(g)$ to ${\bf t}(g)$.}

\item[(G2)] 
A continuous map  ${\bf m}:\mathcal{G}^{(2)}\ap\mathcal{G}$, where $\mathcal{G}^{(2)}:=\{(g,h)\in\mathcal{G}\times\mathcal{G}\mid {\bf s}(g) = {\bf t}(h)\}$ is provided with the induced topology from the  product topology on $\mathcal{G}\times \mathcal{G}$, 
 called a \emph{multiplication} denoted ${\bf m}(g,h)=gh$ and satisfying an associativity relation  in the sense that the product $(gh)k$ is
defined if and only if $ g(hk)$  is defined and in this case we must have $(gh)k = g(hk)$.
%

\item[(G3)] A continuous  embedding ${\bf 1} : M \ap\mathcal{G}$ called \emph{identity section} which satisfies
$g{\bf 1}_x = g$ for all $g\in{\bf s}^{-1}(x)$, and ${\bf 1}_x g = g$ for all $g\in{\bf t}^{-1}(x)$ 
(what in particular implies ${\bf s}\circ{\bf 1}=\id_{M}={\bf t}\circ{\bf 1}$).

\item[(G4)]  A homeomorphism  ${\bf i} : \mathcal{G}\ap \mathcal{G}$,  
denoted ${\bf i}(g)=g^{-1}$ called \emph{inversion}, which satisfies
$gg^{-1}={\bf 1}_{{\bf t}(g)}$, $\;g^{-1}g={\bf 1}_{{\bf s}(g)}$ (what in particular implies ${\bf s}\circ{\bf i}={\bf t}$, $\; {\bf t}\circ {\bf i}={\bf s}$).
\end{itemize}
 The space $ M$ is  called {\it the base} of the groupoid, and $\mathcal{G}$ 
is called the \emph{total space} of the groupoid. 
(See for instance \cite[Def. 1.3]{CraFe}.)

For any $x,y\in M$, we denote $\mathcal{G}(x,-):={\bf s}^{-1}(x)$,  $\mathcal{G}(-,y):={\bf t}^{-1}(y)$, and 
$$\mathcal{G}(x,y):=\{g\in\mathcal{G}\mid {\bf s}(g)=x,\ {\bf t}(g)=y\}=\mathcal{G}(x,-)\cap\mathcal{G}(-,y).$$
The \emph{isotropy group} at $x\in M$ is the set 
$$\mathcal{G}(x)=\{ g\in \mathcal{G}\mid {\bf s}(g)=x={\bf t}(g)\}=\mathcal{G}(x,x)\subseteq\Gc$$
and the {\it  orbit} of $x\in M$ is the set
$$\mathcal{G}.x=\{{\bf t}(g)\mid g\in {\bf s}^{-1}(x)\}={\bf t}(\mathcal{G}(x,-))\subseteq M.$$
For any $g\in \mathcal{G}$, if ${\bf s}(g)=x$ and ${\bf t}(g)=y$, 
then we define its corresponding left translation 
$$L_g: \mathcal{G}(-,x)\ap \mathcal{G}(-,y), \quad h\mapsto gh,$$ and similarly the right translation $R_g:\mathcal{G}(x,-)\ap \mathcal{G}(y,-)$, $h\mapsto hg$. 
Each of these maps $L_g$ and $R_g$ is a  homeomorphism.

The topological groupoid $\Gc\tto M$ is called \emph{{\bf s}-connected} if its {\bf s}-fibers $\Gc(x,-)$ are connected for all $x\in M$, and
\emph{{\bf s}-simply connected} if all its {\bf s}-fibers are connected and simply connected.
 
The topological groupoid is called  \emph{transitive} if the map
 $({\bf s},{\bf t}): \mathcal{G}\ap M\times M$
 is surjective. 
 
 A {\it  topological morphism between  topological groupoids} $ \mathcal{G}\tto M$ and $\mathcal{H}\tto N$ is given by a pair of  continuous maps $\Phi:\mathcal{G}\ap \mathcal{H}$ and $\phi:M\ap N$ which are compatible with the structure maps, that is:
 \begin{itemize}
 	\item 
$\phi({\bf s}(g))={\bf s}(\Phi(g))$, $\;\;\phi({\bf t}(g))={\bf t}(\Phi(g))$  and  $\Phi(g^{-1})=\Phi(g)^{-1}$ for all $g\in \mathcal{G}$;
\item if $(g,g')\in\Gc^{(2)}$ then 
$\Phi(gg')=\Phi(g)\Phi(g')$; 
 \item  $\Phi({\bf 1}_x)={\bf 1}_{\phi(x)}$ for every $x\in M$.
 \end{itemize}
Note that $\phi$ is uniquely determined by $\Phi$ and therefore a morphism between the groupoids $ \mathcal{G}\tto M$ and $\mathcal{H}\tto N$ is given only by 
a continuous map $\Phi:\mathcal{G}\ap \mathcal{H}$ which satisfies the above compatibility conditions.

A {\it subgroupoid} of $\mathcal{G}\tto M$  is a groupoid  $\mathcal{H}\tto N$ such that  $\mathcal{H}\subset \mathcal{G}$ 
and the inclusion $\iota:\mathcal{H}\ap \mathcal{G}$ is a topological morphism of groupoids.  
A subgroupoid $\mathcal{H}\tto N$ of $\mathcal{G}\tto M$ is called a {\it wide subgroupoid}  if $N=M$

\subsection{Some technical lemmas}
We will use the following results which are essentially extracted from the collection of Bourbaki's books. 

\begin{lemma}\label{smooth1}
A continuous left action of a n.n.H. topological group on n.n.H.  topological space $G\times X\to X$, $(g,x)\mapsto g.x$,  
is proper if and only if it satisfies the following condition: 
For every net $\{(g_j,x_j)\}_{j\in J}$ in $G\times X$ for which there exists $\lim\limits_{j\in J}(g_j.x_j,x_j)=:(b,a)\in X\times X$, 
there also exists $\lim\limits_{j\in J}g_j=:g\in G$ and $g.a=b$. 
\end{lemma}

\begin{proof}
See the comment after \cite[Ch. III, \S 4, no. 1, Def. 1]{Bo71}. 
\end{proof}

\begin{lemma}\label{popertop}
Let $G\times X\to X$, $(g,x)\mapsto g.x$ be a continuous  proper action  of n.n.H. topological group on a  n.n.H. topological space.
Then the quotient space $G\setminus X$ is a Hausdorff topological space. Moreover, if $G$ is Hausdorff, then $X$ is also Hausdorff.
\end{lemma}

\begin{proof} 
See \cite[Ch. III, \S 4, no. 2, Prop. 3]{Bo71}.
\end{proof}

As noted in \cite[Ch. III, \S 1, no. 5, Prop. 10]{Bo72} in the case of finite-dimensional manifolds and in \cite[Th. I]{Gl15} for Banach manifolds acted on by Banach-Lie groups, 
the freeness hypothesis in \eqref{smooth2_item1} of the following lemma ensures that 
the tangent map of $\rho(x)$ is injective for every $x\in X$. 

\begin{lemma}\label{smooth2}
Let $G\times X\to X$, $(g,x)\mapsto g.x$ be a smooth action of a Banach-Lie group on a n.n.H. Banach manifold, 
satisfying the following conditions: 
\begin{enumerate}[(a)]
\item\label{smooth2_item1} The action is free and proper. 
\item\label{smooth2_item2} For every $x\in X$ the map $\rho(x)\colon G\to X$, $g\mapsto g.x$ is a split immersion. 
\end{enumerate} 
Then the quotient topological space $G\setminus X$ has the Hausdorff property and has the unique structure of a Banach manifold 
for which the quotient map $q\colon X\to G\setminus X$ is a submersion. 
Moreover $X$ is Hausdorff and $(X,G\setminus X,G,\pi)$ is a  Banach  principal $G$-bundle. 
\end{lemma}

\begin{proof} 
Any Banach-Lie group is Hausdorff by 
\cite[Ch. III, \S 2, no. 6, Prop. 18(a)]{Bo71}. 
Then we can use Lemma~\ref{popertop} to obtain that both $G\setminus X$ and $X$ are Hausdorff. 
Finally, it follows by \cite[Ch. III, \S 1, no. 5, Prop. 10]{Bo72} that $G\setminus X$  has the unique structure of a Banach manifold 
for which the quotient map $q\colon X\to G\setminus X$ is a submersion. 
\end{proof}

\begin{lemma}\label{smooth3}
Let $Z\subseteq X\subseteq Y$ be n.n.H. Banach manifolds. 
If $Z\subseteq Y$ is a submanifold and $X\subseteq Y$ is a submanifold, 
then also $Z\subseteq X$ is a submanifold. 
\end{lemma}

\begin{proof}
It follows by \cite[5.8.5]{Bo09} that the inclusion map $\iota\colon Z\hookrightarrow X$ is smooth. 
Then the map $\iota$ is an immersion as a direct consequence of the hypothesis. 
It remains to show that for every $z\in Z$ the subspace $T_zZ\subseteq T_zX$ is 
split. 
To this end, using that $Z\subseteq Y$  is a submanifold, we can find a closed linear subspace $\Vc\subseteq T_zY$ with 
$T_zZ\oplus\Vc=T_zY$. Since $T_zZ\subseteq T_zX\subseteq T_zY$, it then follows that 
$T_zZ\oplus(\Vc\cap T_zX)=T_zX$, and this completes the proof. 
\end{proof}

\begin{remark}
\normalfont
The proof of Lemma~\ref{smooth3} uses only the fact that the subspace $T_xX\subseteq T_xY$ is closed, 
but not that it is
a split subspace, for $y\in Y$.  
However the splitting condition 
seems to be necessary in order to be able to use \cite[5.8.5]{Bo09}.  
\end{remark}



\section{n.n.H.  Banach-Lie groupoids}\label{Sect3}

\subsection{Definition and basic properties}
The definition of a n.n.H. Banach-Lie groupoid requires some basic facts established in the following proposition. 

\begin{proposition}\label{1} Let $\mathcal{G}\tto M$ be a topological groupoid  satisfying the following conditions: 
\begin{enumerate}
\item $ \mathcal{G}$ is a n.n.H Banach manifold and $M$ is a Banach manifold;
 \item the map ${\bf s}:\mathcal{G}\ap M$ is a submersion;
\item  the map ${\bf i}: \mathcal{G}\ap \mathcal{G}$ is a smooth diffeomorphism.
\end{enumerate}
  \noindent Then ${\bf t}$ is also a  submersion and so for any $x\in M$ each fiber $\mathcal{G}(x,-)$ and $\mathcal{G}(-,x) $ are n.n.H.  Banach submanifolds of $\mathcal{G}$ and $\mathcal{G}(x,-)$ is Hausdorff if and only if  $\mathcal{G}(-,x) $ is Hausdorff.  
  Moreover we have:
\begin{enumerate}[{\rm(i)}]
\item The topological space $\mathcal{G}^{(2)}$ is a n.n.H. Banach submanifold of $\mathcal{G}\times \mathcal{G}$.
\item The map ${\bf 1} : M \ap\mathcal{G}$ is smooth and ${\bf 1}(M)$ is a closed Banach  submanifold of  $\mathcal{G}$.  
\end{enumerate}
  Moreover  
  if  $ \mathcal{G}$ is a Banach manifold than for each $x\in M$, each fiber $\mathcal{G}(x,-)$ and $\mathcal{G}(-,x) $ are   Banach submanifolds of $\mathcal{G}$  and $\mathcal{G}^{(2)}$ is a  Banach submanifold of $\mathcal{G}\times \mathcal{G}$. 
  \end{proposition}
 
 \begin{proof}
 At first, since  ${\bf s}\circ {\bf i}={\bf t}$, from (2) and (3), it follows that ${\bf t}$ is a  submersion. 
It follows by \cite[Ch. II, Prop. 2.2]{La} that the fibers $\mathcal{G}(x,-)$ and $\mathcal{G}(-,x) $ are n.n.H. Banach submanifolds of $\mathcal{G}$ for any $x\in M$. 
Now assume that $\mathcal{G}(x,-)$ is Hausdorff. 
But $\mathcal{G}(-,x)={\bf i}(\mathcal{G}(x,-))$. 
Since ${\bf i}$ is a diffeomorphism of $\mathcal{G}$ this implies that $\mathcal{G}(-,x)$ is also Hausdorff. 
The same argument can be applied for the converse. 
 
  Since both ${\bf s}$ and ${\bf t}$ are submersions, 
   the map $({\bf s},{\bf t}):\mathcal{G}\times \mathcal{G}\ap M\times M$ is a submersion as well. 
Therefore, since the diagonal $\mathcal{D}\subseteq  M\times M$ is a submanifold, it follows that the set 
 $\mathcal{G}^{(2)}=({\bf s},{\bf t})^{-1}(\mathcal{D})$  
is a n.n.H. Banach submanifold of $\mathcal{G}\times \mathcal{G}$. 
If~$\mathcal{G}$ is Hausdorff then so are $\mathcal{G}^{(2)}$, 
$ \mathcal{G}(x,-)$ and $\mathcal{G}(-,x) $ for all $x\in M$.  

  We first assume that both $\mathcal{G}$~and~$M$ are pure Banach manifolds. This is the case for instance if 
  $\mathcal{G}$ is connected, hence so is $M$. 
	Let $\mathbb{G}$~and~$\mathbb{M}$ be  the Banach spaces on which  $\mathcal{G}$~and~$M$ are modeled, respectively. 
	Fix some $x\in  M$ and set $g={\bf 1}_x$  and so ${\bf s}(g)=x$. 
	Since ${\bf s}$ is a  submersion,  we have a decomposition $\mathbb{G}\equiv T_g\mathcal{G}=\ker T_g{\bf s}\oplus \mathbb{F}$ and  $T_g{\bf s}(\mathbb{F})=T_{x} M\equiv \mathbb{M}$. 
	Thus we may assume that $\mathbb{M}$ is a 
	split closed
	subspace of $\mathbb{G}$ and we write $\mathbb{G}=   \mathbb{K}\oplus \mathbb{M}\equiv \mathbb{K}\times \mathbb{M}$. 
	With these conventions and notations,  there exists also a chart  $(U,\phi)$ of $\Gc$ around $g$ and $(U_0,\Phi_0)$ of $M$ around ${\bf s}(g)$ such that ${\bf s}(U)=U_0$, $\phi(U)=V\times W$ where $V$ and $W$ are open sets of $\mathbb{M}$,  $\phi_0 (U_0)= W$ and  $\phi_0\circ {\bf s}\circ \phi$ is the canonical projection on  $W$. 
	It follows that the restriction ${\bf s}_0$ of  ${\bf s}$ to $\phi^{-1}(\{\phi(g)\}\times W)$ is a diffeomorphism onto $U_0$. 
	Since  ${\bf s}\circ{\bf 1}=\id_{ M}$, it follows that  ${\bf 1}_{U_0}: U_0\ap U$ is smooth and $T_x {\bf 1}$ is injective. 
	In particular, $\phi^{-1}(\{\phi(g)\}\times W)={\bf 1}(U_0)$. 
	This shows that ${\bf 1}(M)$ is a submanifold of $\mathcal{G}$.  
	But since $M$  is Hausdorff and the map ${\bf 1}:M\ap \mathcal{G}$ is continuous and injective it follows that its range ${\bf 1}(M)$ is also Hausdorff. 
	The general case  is obtained by application of this result for each connected component of  $\mathcal{G}$. 
  \end{proof}
 
 \begin{definition}\label{BLG}
\normalfont 
A \emph{n.n.H. Banach-Lie  groupoid}  is a topological groupoid  $\mathcal{G}\tto M$ 
satisfying the following conditions:
\begin{itemize}
\item[(BLG1)] $ \mathcal{G}$ is a n.n.H. Banach manifold and $M$ is a Banach manifold.
 
\item[(BLG2)] The map ${\bf s}:\mathcal{G}\ap M$ is a 
submersion. 

\item[(BLG3)]  The map ${\bf i}: \mathcal{G}\ap \mathcal{G}$ is  smooth.  
 
\item[(BLG4)]  The multiplication ${\bf m}:\mathcal{G}^{2}\ap\mathcal{G}$  is smooth.
\end{itemize}
A  \emph{Banach-Lie  groupoid}  is a  n.n.H. Banach-Lie  groupoid  whose total space has the Hausdorff property, 
i.e., is a  Banach manifold. 

A  n.n.H. Banach-Lie  groupoid $\Gc\tto M$ 
is called \emph{pure} if $\Gc$ is a pure n.n.H. Banach manifold and $M$ is a pure Banach manifold.

A  n.n.H. Banach-Lie  groupoid $\Gc\tto M$ is called \emph{split} if for every $x,y\in M$ the set $\Gc(x,y)$ is  submanifold of $\Gc$.

A  {\it Banach-Lie morphism} between the n.n.H Banach-Lie groupoids $ \mathcal{G}\tto M$ and $\mathcal{H}\tto N$ is a  topological morphism $\Phi:\mathcal{G}\ap \mathcal{H}$ which is a smooth map. 
\end{definition}

 \subsection{Examples}
 We will adapt to our context some classical examples of  
 finite-dimensional Lie groupoids.
 
\subsubsection{Banach-Lie groups}
Any Banach-Lie group $G$ is a Banach-Lie groupoid:  
the set of  arrows $\mathcal{G}$ is the set $G$ and the set of objects $M$ is reduced to the singleton~$\{\1\}$, 
where $\1\in G$ is the unit element.

\subsubsection{Banach-Lie pair groupoid}
Given a Banach manifold $M$, let $\mathcal{G}:=M\times M$ and  let ${\bf s}$ and ${\bf t}$ be the Cartesian projections of $M\times M$ on the first and the second factor, respectively. 
The multiplication   map ${\bf m}$ and the inverse ${\bf i}$ are respectively ${\bf m}((x,y),(y,z))=(x,z)$ and ${\bf i}(x,y)=(y,x)$. Finally the map ${\bf 1}$ is  ${\bf 1}(x)=(x,x)$. 
We thus obtain a Banach-Lie groupoid $M\times M\tto M$.
 
 
\subsubsection{General linear Banach-Lie  groupoids}
Let $\pi:\mathcal{A}\ap M$ be a Banach vector bundle.
 The general linear Banach  groupoid $\mathcal{GL}(\mathcal{A})\tto M$ is the  Banach groupoid such that 
 $\mathcal{GL}(\mathcal{A})$ is the set of  linear isomorphisms $g:\mathcal{A}_x\ap\mathcal{A}_y$ between each pair of fibers $(\mathcal{A}_x,\mathcal{A}_y)$. The source  map and the target map are obvious and the multiplication is the composition of linear isomorphisms and ${\bf 1}_x=Id_{\mathcal{A}_x}$.

 \subsubsection{Disjoint union of n.n.H. Banach-Lie  groupoids}\label{dis}
 Let $\{\mathcal{G}_\lambda \tto M_\lambda\}_{\lambda\in \L}$ be a family of n.n.H. Banach-Lie groupoids. 
 We denote by 
 $\mathcal{G}:=\bigsqcup\limits_{\lambda\in \L}\mathcal{G}_\lambda$
 and $M:=\bigsqcup\limits_{\lambda\in \L}M_\lambda$. 
 Since we consider here disjoint unions, the structure of
the n.n.H. Banach-Lie groupoid $\mathcal{G}\tto M$ is clearly defined by the collection of structure maps 
for each particular $\mathcal{G}_\lambda \tto M_\lambda$ for $\lambda\in \L$. 
For example if we consider a finite family of Banach bundles $\mathcal{A}_i \rightarrow M_i$, for $i = 1,\dots,n$, 
then we have the natural structure of
a n.n.H. Banach-Lie groupoid  $\bigsqcup\limits_{i=1}^n \mathcal{GL}(\mathcal{A}_i) \tto  \bigsqcup\limits_{i=1}^n  M_i$. 
 More generally, given any Banach-Lie n.n.H. groupoid $\mathcal{G}\tto M$, 
if $M=\bigsqcup\limits_{\lambda\in L} M_\lambda$ is a partition of $M$ into open $\Gc$-invariant submanifolds, 
then the corresponding groupoids $\mathcal{G}_\lambda\tto M_\lambda$ obtained by restriction 
are n.n.H. Banach-Lie groupoids and  
$\mathcal{G}\tto M$ is the disjoint union of 
the family $\{\mathcal{G}_\lambda \tto M_\lambda\}_{\lambda\in \L}$.

\subsubsection{Action of a Banach-Lie group}
\label{action}
To each  smooth action $A: G\times M\ap M$, $(g,x)\mapsto g.x$, of a Banach-Lie group on a Banach manifold  there is associated a Banach-Lie groupoid $\mathcal{G}\tto M$ defined in the following way:
\begin{itemize}
\item $\mathcal{G}:=M\times G$; 

\item ${\bf s}(x,g):=x$ and ${\bf t}(x,g):= A(g,x)=g.x$;

\item if $y=g.x$ then ${\bf m}((y,h),(x,g)):=(x,hg)$;

\item ${\bf i}(x,g):=(g.x,g^{-1})$; 
 
\item ${\bf 1}_x=(x,\1)$.
 \end{itemize}
It is easily seen that for any $x_0,y_0\in M$ one has 
$$\Gc(x_0,y_0)=\{x_0\}\times\{g\in G\mid g.x_0=y_0\}$$
and in particular the isotropy group at $x_0$ is 
$$\Gc(x_0)=\{x_0\}\times G(x_0)$$
where $G(x_0):=\{g\in G\mid g.x_0=x_0\}$. 

It is worth pointing out that  the above groupoid needs not be split. 
To obtain a specific example in this connection, let $\Xc$ be any real Banach space with a closed linear subspace $\Xc_0$ 
which fails to be a split subspace of~$\Xc$. 
(For instance the space of all bounded sequences of real numbers $\Xc:=\ell^\infty_{\RR}(\NN)$ with its subspace $\Xc_0$ consisting of the sequences that converge to~$0$.) 
Then the abelian Banach-Lie group $G:=(\Xc,+)$ acts smoothly transitively on the Banach manifold $M:=\Xc/\Xc_0$ by  $A(g,x+\Xc_0):=g+x+\Xc_0$ for all $g,x\in\Xc$. 
If we define the groupoid $\Gc:=M\times G\tto M$ as above, 
then the isotropy group at the point $\Xc_0\in M$, that is,   
$$\Gc(\Xc_0)=\{\Xc_0\}\times\Xc_0\subseteq\Gc,$$
is not a submanifold of $\Gc=M\times \Xc$ since $\Xc_0$ 
fails to be a split subspace of~$\Xc$. 
 
 \subsubsection{The gauge groupoid of a principal bundle}\label{gauge}
A principal Banach bundle   is a locally trivial fibration $\pi: P\ap M$ over a connected Banach manifold whose typical fiber is a Banach-Lie group $G$.
 We have then a right action of $G$ on $P$ whose corresponding quotient $P/G$ is canonically diffeomorphic to~$M$. 
 We get a right diagonal  action of $G$ on $P\times P$ in an evident way. 
The gauge groupoid is the set of orbits of this action, that is, the quotient $\Gc:=(P\times P)/G$ provided with the quotient topology. 
The source (resp. target) of the equivalence class of $(v,u)$ is  $\pi(u)$ (resp. $\pi(v)$). 
The composition of the class  of $(v,u)$ and of $(v',u')$ is the  equivalence class of $(w',u)$ and 
the inverse of the equivalence class of $(v,u)$ is the equivalence class of $(u,v)$. 
(See  \cite{Ma} for more details.)

\subsubsection{Fundamental groupoid of a  Banach manifold} \label{funfgrd} 
Recall that for any  topological space $M$ there is its corresponding fundamental 
groupoid $\Pi(M)\tto M$
where  $\Pi(M)$ is the set  of homotopy classes of continuous paths with fixed end
points.
The source map (resp. target map) is the map which to a homotopy class $[\g]$ of a path $\g$  associates its origin ${\bf s}(\g)$ (resp. its end ${\bf t}(\g)$). 
The multiplication is obtained by the concatenation of paths: if $\g$ and $\g'$ are two paths defined on $[0,1]$ the concatenation $\g\star\g'$ is the path defined by
$\g\star\g'(t)=\g(2t)$for $ 0\leq t\leq \frac{1}{2}$ and $\g\star\g'(t)=\g'(2t-1)$  for $ \frac{1}{2}\leq t\leq 1$. It is compatible with homotopy equivalence. The inverse of a homotopy class of a path $\g:[0,1]\to M$ is the homotopy class of the path $\g^{-1}: t\mapsto \g(1-t)$.

When $M$ is a connected  Banach manifold, each source fiber  is a universal covering of $M$. 
In particular such a fiber is a Banach manifold. 
Moreover it is also a principal bundle over $M$ whose structural group is the fundamental group $\pi_1(M)$. 
Then the gauge groupoid of this principal bundle can be identified with $\Pi(M)$ and so  we get a structure of n.n.H. Banach-Lie groupoid structure on $\Pi(M)$.

\subsection{Properties of orbits}\label{orbprop}
The purpose of this subsection is to show the following results for Banach-Lie   groupoids which are an adaptation of similar  classical results in the finite dimensional case.
 
\begin{theorem}\label{5.4} 
Let  $\mathcal{G}\tto M$ be a n.n.H.  Banach-Lie  groupoid  and define
$$(\forall x\in M)\quad {\bf t}_x:={\bf t}\vert_{\Gc(x,-)}\colon \Gc(x,-)\to M.$$
The following assertions hold.
 \begin{enumerate}[{\rm(i)}]
 \item  \label{smooth4_item0} Let $N$  be  a connected component of $M$.  Then $\mathcal{G}^{\bf s}_{{N}}={\bf s}^{-1}({N})$ and   $\mathcal{G}^{\bf t}_{{N}}={\bf t}^{-1}({N})$  are   pure n.n.H. Banach manifolds  and 
 submanifolds 
 of~$\mathcal{G}$. 
  \item \label{smooth4_item1}  For all $x\in M$ and $y\in \mathcal{G}.x$ the set $\mathcal {G}(x,y)$ is a closed 
  submanifold of $\mathcal{G}$. 
  In particular the  isotropy group   $\mathcal{G}(x)$ is a Banach-Lie group  and 
  $T_{{\bf 1}_x}(\Gc(x))=\ker T_{{\bf 1}_x}{\bf s}\cap \ker T_{{\bf 1}_x}{\bf t} $.
   \item\label{smooth4_item2}   
   For every $x\in M$ for which $\Gc(x,y)$ is a split submanifold of $\Gc$ for all  $y\in\Gc.x={\bf t}_x
   (\mathcal{G}(x,-))$, 
   the orbit $\Gc.x$ is a  pure Banach 
   manifold whose inclusion map $\Gc.x\hookrightarrow M$ is a weak immersion and ${\bf t}_x
   \colon \mathcal{G}(x,-)\ap \mathcal{G}.x$ is  a Banach principal $\Gc(x)$-bundle. 
      \end{enumerate}
 \end{theorem}

  \begin{proof}   
  	We first prove Assertion~\eqref{smooth4_item0}. 
  For any connected component ${N}$ of $M$,  
	the set $\mathcal{G}^{\bf s}_{{N}}$ is open and closed in $\Gc$ hence it is the union of some connected components 
	of~$\mathcal{G}$.   
	Consider any connected component  $\mathcal{G}_\alpha$ of $\mathcal{G}$ which is contained in $\mathcal{G}^{\bf s}_{{N}}$. 
	One clearly has ${\bf s}(\mathcal{G}_{\alpha})={N}$. 
	Now for $x\in {N}$, each fiber $\mathcal{G}_\alpha(x,-)$ of  the restriction of ${\bf s}$ to $\mathcal{G}_\alpha$ is an open submanifold of $\mathcal{G}(x,-)$. 
	Denote by $\mathbb{N}$  and $\mathbb{G}_\alpha$ the Banach spaces on which ${N}$ and  $ \mathcal{G}_\alpha$  are  modeled respectively. 
	Since ${\bf s}$ is a  submersion, $\mathbb{G}_\alpha$ is isomorphic to  $T_g \mathcal{G}(x,-)\oplus \mathbb{N}$ for any $g\in \mathcal{G}_\alpha(x,-)$ and $x\in {N}$. 
	It follows easily that all connected components of $\mathcal{G}$ which are contained in $\mathcal{G}^{\bf s}_{{N}}$ are modeled on the same Banach space.
	 Thus  $\mathcal{G}^{\bf s}_{{N}}$ is a pure n.n.H.  Banach manifold and also a  Banach submanifold of $\mathcal{G}$. 
	 The same result for $\mathcal{G}^{\bf t}_{{N}}$ can be proved by similar arguments. 
	This ends the proof of Assertion~\eqref{smooth4_item0}. 
  
To prove  Assertion~\eqref{smooth4_item1} we adapt the method of proof of \cite[Th. 5.4]{MoMr}. 
To this end we will show that defining 
  $$\Delta_g=\ker T_g{\bf s}\cap \ker T_g{\bf t} 
  \quad \text{for all }g\in \mathcal{G}(x,-),$$
one obtains a Banach subbundle  $\Delta:=\bigcup\limits_{g\in\mathcal{G}(x,-)}\Delta_g\subseteq T\mathcal{G}(x,-)$ which is integrable. 
  Given any $g\in\mathcal{G}(x,-)$ the left translation 
 $$L_g\colon \mathcal{G}(-,x)\to \mathcal{G}(-,{\bf t}(g)), 
 \quad L_g(h)=gh$$ 
  is a diffeomorphism just as~\eqref{diffeo_R}
  and moreover ${\bf s}\circ L_g ={\bf s}\vert_{\mathcal{G}(-,x)}$, hence 
  $$TL_g(\Delta_{{\bf 1}_x})=\Delta_g.$$
 Note that $\Delta_g$ is a closed subspace of $T_g(\mathcal{G}(-,x))$.  
 The map  
 $$\Phi: \mathcal{G}(x,-)\times \Delta_{{\bf 1}_x}\ap T\mathcal{G}(x,-),\quad \Phi(g,v)=TL_g(v)$$  
 is an injective  morphism over $\id_{\mathcal{G}(x,-)}$ from  the trivial bundle $\mathcal{G}(x,-)\times \Delta_{{\bf 1}_x}\ap \mathcal{G}(x,-)$ to $T\mathcal{G}(x,-)\ap \mathcal{G}(x,-)$, whose range is the distribution $\Delta$.
 Thus  $\Delta$ defines a smooth trivial Banach bundle which is a Banach subbundle of $T\mathcal{G}(x,-)$. 
 For any $v\in  \Delta_{{\bf 1}_x}$ denote by $X_v$ the vector field on $  \mathcal{G}(x,-)$ defined by $X_v(g)=TL_g(v)$. 
 Then the set $$\Gamma(\Delta)=\{X_v\mid v\in  \Delta_{{\bf 1}_x}\}$$ generates the distribution $\Delta$. 
 Now any integral curve $\gamma: {]-\varepsilon,\varepsilon[} \ap \mathcal{G}(x,-)$  of $X\in \Gamma(\Delta)$ with $\gamma(0)=g$ is  also contained in $\mathcal{G}(-, {\bf t}(g))$. 
 
 It follows that the Lie bracket $[X,Y](g)$  of vector fields $X$ and $Y$ in $\Gamma(\Delta)$ is tangent to $\mathcal{G}(-,{\bf t}(g))$ and so $\Gamma(\Delta)$ is stable under Lie bracket. 
 From \cite[Th. 4]{Pe}\footnote{That theorem was proved in the context of (Hausdorff) Banach manifolds but all the arguments used in the proof of this result are based on local charts so the theorem also holds for n.n.H. Banach manifolds.}, 
 $\Delta$ is integrable, i.e., 
 there exists a partition of $\mathcal{G}(x,-)$ into immersed n.n.H. Banach manifolds and each one  is modelled on the Banach space $\Delta_{{\bf 1}_x}$. 
 Moreover since $\Delta_g=\ker T_g{\bf t}_x$, the maximal leaf through $g\in \mathcal{G}(x,-)$ is a connected component of ${\bf t}^{-1}({\bf t}(g))$ and so is a closed subset of the  Banach manifold $ \mathcal{G}(x,-)$. 
 In particular  the isotropy group $\mathcal{G}({{\bf t}(g)})$ is an union of such leaves and so is a closed immersed n.n.H. Banach manifold. 
 But from Remark \ref{nnH}  each point in  $\mathcal{G}({{\bf t}(g)})$ is closed thus it follows from  \cite[Ch. III, \S 2, no. 6, Prop. 18(a)]{Bo71} that $\mathcal{G}({{\bf t}(g)})$ is in fact Hausdorff.  
 This implies that $\mathcal{G}({{\bf t}(g)})$ has the structure of a Banach-Lie group.
 
Finally, the left translation $L_g$ is a diffeomorphism from $\mathcal{G}(-,x)$ to $\mathcal{G}(-, {\bf t}(g))$ and  $TL_g(\Delta_{{\bf 1}_x})=\Delta_g$, hence it follows that the restriction of  $L_g$ to $\mathcal{G}(x)$ is a diffeomorphism onto $\mathcal{G}({{\bf t}(g)})$. We conclude that each fiber of ${\bf t}_x$  is  diffeomorphic to $\mathcal{G}(x)$.  
This proves the first part of the assertion and 
 $T_{{\bf 1}_x}(\Gc(x))=\ker T_{{\bf 1}_x}{\bf s}\cap \ker T_{{\bf 1}_x}{\bf t} $.
This ends the proof of  Assertion~\eqref{smooth4_item1}. 

For the proof of Assertion \eqref{smooth4_item2}  we consider the map 
$$\Psi\colon \Gc(x,-)/\Gc(x)\to \Gc.x ,\quad g\cdot\Gc(x)\mapsto {\bf t}(g).$$
It is easy to prove that  $\Psi$ is bijective. 
Therefore to prove the first part of Assertion~\eqref{smooth4_item2}, it suffices to show that the quotient set $\Gc(x,-)/\Gc(x)$ has the structure of a smooth manifold for which 
the map $\Psi$ is smooth from $\Gc(x,-)/\Gc(x)$ into~$M$.

Since the map ${\bf s}$ is a submersion, 
it follows by Proposition~\ref{1}
that the set $\Gc(x,-)$ is a n.n.H. submanifold of~$\Gc$. 
On the other hand, the quotient set $\Gc(x,-)/\Gc(x)$ is the set of orbits of the right group action 
$$\Rc\colon \Gc(x,-)\times\Gc(x)\to\Gc(x,-)$$
defined by  
  $$(h,g)\mapsto \mathcal{R}(h,g)=hg.$$
This is a smooth action of a Banach-Lie group on a n.n.H. Banach manifold since the groupoid multiplication is smooth.  
The action $\Rc$ is free since every element of $\Gc(x,-)$ has an inverse in $\Gc$. 
Using Lemma~\ref{smooth1} and continuity of multiplication and inversion maps of the groupoid~$\Gc$, 
it also follows that the action~$\Rc$ is proper. 
It then follows by Lemma~\ref{popertop} that both 
$\Gc(x,-)$ and $\Gc(x,-)/\Gc(x)$ are Hausdorff. 

We now check that 
for every $h\in \Gc(x,-)$ the map $r_h\colon\Gc(x)\to\Gc(x,-)$, $g\mapsto gh$, 
is an immersion. 
We have already seen that $r_h$ 
defines a diffeomorphism $\Gc(x)\to\Gc(x,y)$ where $y={\bf t}(h)$. 
One has 
$$\Gc(x,y)\subseteq\Gc(x,-)\subseteq\Gc$$
where we know that $\Gc(x,y)\subseteq\Gc$ is a  submanifold and 
$\Gc(x,-)\subseteq\Gc$ is a submanifold, hence $\Gc(x,y)\subseteq\Gc(x,-)$ 
is a submanifold by Lemma~\ref{smooth3}. 
Thus the above map $r_h$ is a diffeomorphism of $\Gc(x)$ onto  the submanifold $\Gc(x,y)$ of $\Gc(x,-)$, 
and in particular $r_h$ is an immersion. 

Thus we can apply Lemma~\ref{smooth2} to the left action canonically associated to the right action~$\Rc$, 
$$\Gc(x)\times\Gc(x,-)\to\Gc(x,-),\quad (g,h)\mapsto\Rc(h,g^{-1})$$
First,  it follows that the set $\Gc(x,-)/\Gc(x)$ has the unique structure of a smooth manifold for which 
the quotient map $q\colon \Gc(x,-)\to \Gc(x,-)/\Gc(x)$ is a submersion. 

Moreover, one has $\Psi\circ q={\bf t}\vert_{\Gc(x,-)}\colon\Gc(x,-)\to\Gc$, which is a smooth map. 
Since $q$ is a submersion, it then follows that $\Psi\colon \Gc(x,-)/\Gc(x)\to\Gc$ is smooth. 
(See for instance \cite[5.9.5]{Bo09}.)  
This implies that  the inclusion map $\Gc.x\hookrightarrow M$ is smooth. 
Its tangent map at every point of $\Gc.x$ is injective by \cite[Th. I]{Gl15}.
Finally, it follows by Lemma~\ref{smooth2} that $\Gc(x,-)$  ${\bf t}_x: \Gc(x,-)\ap \Gc.x$ is a Banach principal $\Gc(x)$-bundle.  
\end{proof}

\begin{remark}\label{immersedorbit}
	\normalfont 
	Note that Theorem \ref{5.4}\eqref{smooth4_item1}  ensures {\bf only} that   $\mathcal {G}(x,y)$  is a closed  immersed submanifold of $\mathcal{G}$ but a priori    {\bf not a submanifold} as it is pointed out in \ref{action}. 
	This problem is illustrated by the following examples.  
	
	In the same way  Theorem \ref{5.4}\eqref{smooth4_item2}  ensures {\bf only} that each orbit $\Gc.x$  is an immersed submanifold of $M$ but a priori not a {\bf closed} submanifold and in particular of course {\bf not a submanifold}.
\end{remark}

\begin{example}
\normalfont 
If $\Gc\tto M$ is a finite-dimensional Lie groupoid, then  
it is well known that 
that $\Gc(x,y)$ is a submanifold of $\Gc$ for every $x\in M$ and $y\in\Gc.x$. 
See for instance \cite[Cor. 1.4.11]{Ma} and \cite[Th. 5.4]{MoMr}. Thus $\Gc\tto M$  is split.
\end{example}

  \begin{example}
  \normalfont
Let $\Gc\tto M$ be the Banach-Lie groupoid associated as in~\ref{action} to the smooth action $A\colon G\times M\to M$ of a Banach-Lie group on a Banach manifold. 
Let $x\in M$ with its isotropy group $G(x):=\{g\in G\mid A(g,x)=x\}$ and with its orbit $\Oc(x):=\{A(g,x)\mid g\in G\}$. 
One has 
$$\Gc(x)=\{(h,x)\in G\times M\mid h\in G(x)\}=G(x)\times \{x\}
\subseteq G\times M=\Gc$$
hence $\Gc(x)$ is a submanifold of $\Gc$ if and only if $G(x)$ is a submanifold of $G$, that is, if and only if $G(x)$ is a Banach-Lie subgroup of $G$. 

For any $y\in\Oc(x)$ and $g_0\in G$ with $A(g_0,x)=y$ one has 
$$\begin{aligned}
\Gc(x,y)
&=\{(g,x)\in G\times M\mid A(g,x)=y\}
=\{(g_0h,x)\in G\times M\mid h\in G(x)\} \\
&=g_0G(x)\times\{x\}
\subseteq G\times M=\Gc.
\end{aligned}$$
This shows that the condition that 
$G(x)$ is a Banach-Lie subgroup of $G$ for all $x\in M$ is equivalent to the assumption that $\mathcal{G}$ is split.  
Under this assumption
 the conclusion of Theorem~\ref{5.4}\eqref{smooth4_item2} recovers the classical fact that if $G(x)$ is a Banach-Lie subgroup of the Banach-Lie group $G$, then $G/G(x)$ is a smooth homogeneous Banach manifold (cf. for instance \cite[Th. 4.19]{B06}).  
 
For the above Banach-Lie groupoid $\Gc\tto M$, 
 Theorem~\ref{5.4}\eqref{smooth4_item2} says that $\Gc(x,y)$ is a closed submanifold of $\Gc$ if $y\in\Oc(x)$. 
 By the above description of $\Gc(x,y)$, 
 this property  
is equivalent to the fact that the isotropy group $G(x)$ is a closed submanifold of $G$, that is, $G(x)$ is a closed subset of $G$ and $G(x)$ has the structure of a Banach manifold for which the inclusion map $G(x)\hookrightarrow G$ is an immersion  
as in Subsection~\ref{subsect-nnH}. 
For the sake of completeness, we note that this conclusion 
(which {\bf does not} mean that $G(x)$ is a Banach-Lie subgroup of $G$ unless $\dim G<\infty$, cf. \cite[Ex. II.11]{Ho75}) can also be derived from an infinite-dimensional version of Cartan's theorem on closed subgroups \cite[Cor. 3.7]{B06}, 
since $G(x)$ is a closed subgroup of the Banach-Lie group~$G$. 
\end{example}

  \section{Banach-Lie algebroids}\label{Sect4}
	
   \subsection{Structure of Banach-Lie algebroid} 
Let $\pi:\mathcal{A}\rightarrow M$ be a Banach vector bundle on a Banach manifold.

\begin{definition}
\label{D_AnchoredBanachBundle} 
\normalfont
A morphism of vector bundles $\rho:\mathcal{A}\rightarrow TM$ covering identity is
called an \emph{anchor}. 
The structure 
$(\mathcal{A},\pi,M,\rho)$ is then called 
an \emph{anchored Banach bundle}.

The above morphism $\rho$ induces a map $ \G(\mathcal{A})\to\G(M) $, again denoted by $\r$,
 defined on any section $s\in\G(\mathcal{A})$ by $(\r(s))(x):=(\r\circ s)(x)$  for every $x\in M$. 
\end{definition}

\begin{remark}
\normalfont
 If one has a local trivialization of the vector bundle $\pi\colon\Ac\to M$ in which the manifold $M$ is modeled on a Banach space $\mathbb{M}$ and the typical fiber of $\pi$ is a Banach space $\mathbb{A}$, then the anchor $\rho$ has the local expression  
$$
\r(x,u)\equiv (x,u) \mapsto (x,\r(x)(u))
$$
where $\r\colon U\to L(\mathbb{A},\mathbb{M})$ is a smooth map. 
\end{remark}

\begin{definition}
	\normalfont
A \emph{Lie bracket} on the anchored Banach bundle 
$(\mathcal{A},\pi,M,\rho)$ is a skew-symmetric $\RR$-bilinear map
 $[\cdot,\cdot]_{\mathcal{A}}:\Gamma(\mathcal{A})\times \Gamma(\mathcal{A})\ap \Gamma(\mathcal{A})$ satisfying the following conditions: 
 \begin{enumerate}
 \item Leibniz property:  $[\s_1,f\s_2]_{\mathcal{A}}=f[\s_1,\s_2]_{\mathcal{A}}+df(\rho(\s_1))\s_2$ 
 for all $f\in\Ci(M)$ and $\s_1, \s_2\in \Gamma(\mathcal{A})$.
 \item Jacobi identity:   $[\s_1,[\s_2,\s_3]_{\mathcal{A}}]_{\mathcal{A}}+[\s_2,[\s_3,\s_1]_{\mathcal{A}}]_{\mathcal{A}}+[\s_3,[\s_1,\s_2]_{\mathcal{A}}]_{\mathcal{A}}=0$ 
 for all $\s_1,\s_2,\s_3\in \Gamma({\mathcal{A}})$.
 \end{enumerate}
\end{definition}

\begin{remark}\label{LBA_early}
	\normalfont
 An anchored Banach bundle provided with a Lie bracket $[\cdot,\cdot]_{\mathcal{A}}$ as above was sometimes called  a Banach-Lie algebroid in the earlier literature 
 (for instance in \cite{An11})  
 and was denoted $({\mathcal{A}},M,\rho,[\cdot,\cdot]_{\mathcal{A}})$ if $\rho$ satisfies:
 $$\rho([\s_1,\s_2]_{\mathcal{A}})=[\rho(\s_1),\rho(\s_2)]
 \text{ for all }\s_1, \s_2\in\Gamma({\mathcal{A}})$$
 where $[\cdot,\cdot]$ denotes the usual Lie bracket of vector fields on~$M$.  

 If $({\mathcal{A}},M,\rho,[\cdot,\cdot]_{\mathcal{A}})$ is a Banach-Lie algebroid in this sense, then  $(\Gamma({\mathcal{A}}),[\cdot,\cdot]_{\mathcal{A}})$ is a Lie algebra and $\rho$ is 
a Lie algebra morphism from $(\Gamma({\mathcal{A}}),[\cdot,\cdot]_{\mathcal{A}})$ to the Lie algebra $(\Gamma(M),[\cdot,\cdot])$ of smooth vector fields on~$M$.
\end{remark} 

In this paper we prefer the definition of a Banach-Lie algebroid (Definition~\ref{LBA}) which, in addition to the Lie algebra morphism property of the anchor map from Remark~\ref{LBA_early}, 
also involves the localizability property that we will discuss right now. 

\begin{definition}
	\normalfont
A Lie  bracket  $[\cdot,\cdot]_{\mathcal{A}}$  on an anchored bundle $(\mathcal{A}, \pi, M, \r)$ respects the sheaf of sections of $\pi:\mathcal{A}\ap M$ or, for short,  is \emph{localizable}  (see for instance \cite{Mar}), if 
for every nonempty open subset $U\subseteq M$ one has a Lie bracket $[\cdot,\cdot]_U$ on the anchored bundle 
$(\mathcal{A}\vert_U, \pi\vert_{\mathcal{A}\vert_U}, U, \r\vert_{\mathcal{A}\vert_U}\colon \mathcal{A}\vert_U\to TU)$ satisfying the following conditions: 
\begin{enumerate}[(i)]
	\item For $U=M$ one has $[\cdot,\cdot]_M=[\cdot,\cdot]_{\Ac}$. 
	\item For any open subsets $V\subseteq U\subseteq M$ 
	and all $\s_1,\s_2\in \G(\mathcal{A}\vert_U)$ one has 
	$$[{\s_1}\vert_{V},{\s_2}\vert_{V}]_V=([\s_1,\s_2]_U)\vert_{ V}.$$  
\end{enumerate}
\end{definition}

\begin{remark}\label{smooth_reg}
\normalfont
In finite dimensions it is well known that every
Lie  bracket  $[\cdot,\cdot]_{\mathcal{A}}$  on an anchored bundle $(\mathcal{A}, \pi, M, \r)$ is localizable.  
(See for instance \cite{Mar}.) 
To extend this fact to infinite dimensions, one needs the following notion: 
A \emph{bump function} on a Banach space $X$ is 
a smooth function on $\varphi\colon X\to \RR$ whose support is a bounded nonempty subset of $X$. 
See for instance \cite[Ch. III, \S 14]{KrMi} for background information on this notion. 

We claim that if the Banach space $X$ admits bump functions, then for every $x_0\in X$ there exist a bump function $\varphi_{x_0}\colon X\to\RR$ and an open neighborhood~$V$ of $x_0$ with $\varphi_{x_0}(x)=1$ for every $x\in V$. 
Without loss of generality let us chose a bump function $\varphi\colon X\to \RR$ such that $\varphi(x_0)=1$. 
For convenience we will fix open neighborhoods $(1/2,3/2)\subset(1/3,5/3)$ of $1\in\RR$. 
Then there exist open neighborhoods $V\subseteq W$ of $x_0\in X$ with $\varphi(V)\subseteq(1/3,5/3)$ and $\varphi(W)\subseteq(1/2,3/2)$. 
If $\psi\colon\RR\to\RR$ is any smooth function satisfying $\psi(t)=0$ if $t\in\RR\setminus(1/2,3/2)$ and $\psi(t)=1$ for every $t\in[1/3,5/3]$, 
then $\varphi_{x_0}:=\psi\circ\varphi\colon X\to\RR$ is a smooth function satisfying $\varphi_{x_0}(x)=1$ for every $x\in V$. 
Moreover, using that $0\not\in(1/3,3/4)$,  
it is straightforward to check that the support of $\varphi_{x_0}$ is contained in the support of $\varphi$, hence $\varphi_{x_0}$ is a bump function as claimed.  

A Banach manifold $M$ is called \emph{smoothly regular} if 
the model space of $M$ at every point $x\in M$ is a Banach space that admits bump functions. 
By the same arguments as in finite dimensions and using the above observation on bump functions that are constant on a neighborhood of any given point, one can prove that if the Banach manifold $M$ is smoothly regular then any Lie bracket  $[\cdot,\cdot]_{\mathcal{A}}$  on an anchored bundle $(\mathcal{A},\pi,M,\r)$ is localizable.
(See \cite[Prop. 3.6]{CaPe} and also \cite{Pe}.) 

If $M$ is not smoothly regular, we cannot prove that any Lie bracket is localizable. 
Unfortunately in the Banach framework, we have {\bf no example of  Lie algebroid} whose Lie bracket is not localizable. 
Moreover we will see later that if a Banach-Lie algebroid is integrable then its Lie bracket is localizable. 
Therefore this condition is necessary in order to find conditions under which a Banach-Lie algebroid is integrable.  
\end{remark}

For these reasons we make the following definition 
(cf. \cite{Pe} and \cite{CaPe}).  

\begin{definition}\label{LBA} 
\normalfont
The structure of an anchored Banach bundled endowed with a Lie bracket  $({\mathcal{A}},{M},\rho,[\cdot,\cdot]_{\mathcal{A}})$  is called a \emph{Banach-Lie algebroid} 
if $[\cdot,\cdot]_{\mathcal{A}}$ is a localizable Lie bracket and $\rho$ is a Lie algebra morphism.

When  $\pi\colon{\mathcal{A}}\ap M$ is a pure Banach bundle then $({\mathcal{A}},{M},\rho,[\cdot,\cdot]_{\mathcal{A}})$ will be called a \emph{pure Banach-Lie algebroid} which corresponds to  the definition of Banach-Lie algebroid given in \cite{Pe} and \cite{CaPe}. 

Consider two Lie algebroids  $({\mathcal{A}},{M},\rho,[\cdot,\cdot]_{\mathcal{A}})$ and $({\mathcal{A}'},{M}',\rho',[\cdot,\cdot]_{\mathcal{A}'})$. 
A Banach bundle morphism $\Phi\colon {\mathcal{A}}\ap {\mathcal{A}'}$ over a map $\phi\colon M\ap M'$ is called a \emph{Banach-Lie algebroid morphism} if it satisfies the conditions 
$$T\phi\circ\rho=\rho'\circ \Phi
\text{ and }
\Phi ([\s_1,\s_2]_{\mathcal{A}})=[\Phi(\s_1),\Phi(\s_2)]_{\mathcal{A}'}
\text{ for all }\s_1,\s_2\in\G(\mathcal{A}).$$
\end{definition}

The following notion of admissibility is used in the proof of Lemma~\ref{linkorbit} below, 
and it thus plays a key role for Theorem~\ref{linkorbitG}. 

\begin{definition}
	\label{LBA-orbit} 
	\normalfont
If $({\mathcal{A}},{M},\rho,[\cdot,\cdot]_{\mathcal{A}})$ is a Banach-Lie algebroid,  
then a piecewise smooth curve $c\colon\RR\supset[a,b]\ap {M}$ is called an \emph{${\mathcal{A}}$-admissible curve} if there exists piecewise smooth curve $\alpha \colon[a,b]\ap \mathcal{A}$ with $\pi(\alpha)=c$ and $\rho( \alpha(t))=\dot{c}(t)$ for each $t\in [a,b]$ at which $c$ is smooth.  
In these conditions $\alpha$ is called an \emph{${\mathcal{A}}$-lift} of $c$ 
and more generally $\alpha$ is called an \emph{$\mathcal{A}$-path}. 
We define 
a binary relation on $M$ associated to $({\mathcal{A}},M,\rho,[\cdot,\cdot]_{\mathcal{A}})$ 
in the following way:
$$x\sim y
\Leftrightarrow
\text{there exists }\gamma\colon[a,b] \ap {M}
\text{ ${\mathcal{A}}$-admissible with } \g(a)=x\text{ and }\g(b)=y.
$$
Since any concatenation of two piecewise smooth paths is in turn piecewise smooth, the above is an equivalence relation, 
and the equivalence class of $x$ is called the \emph{${\mathcal{A}}$-orbit of $x$} 
with respect to the Banach-Lie algebroid $({\mathcal{A}},{M},\rho,[\cdot,\cdot]_{\mathcal{A}})$.
\end{definition}

\begin{lemma}\label{reparam}
Let $\alpha$ be an ${\mathcal A}$-lift of the curve $c$, both these
curves being defined on some interval $J$.
Then for any smooth map $\phi\colon J\to J$, the curve
$(\alpha\circ\phi)\dot{\phi}$ is an ${\mathcal A}$-lift of $c\circ\phi$.
\end{lemma}

\begin{proof}
Since $\alpha$ is an ${\mathcal A}$-lift of $c$, one has
$c=\pi\circ\alpha$ and $\rho\circ\alpha=\dot{c}$,
and these equalities imply:
\begin{enumerate}
\item $\pi\circ(\alpha\circ\phi)=c\circ\phi$;
\item $\rho\circ(\alpha\circ\phi)=\dot{c}\circ\phi$.
\end{enumerate}
Now one obtains
$(c\circ\phi)^{\textstyle\cdot}
=(\dot{c}\circ\phi)\dot{\phi}
                       =(\rho\circ(\alpha\circ\phi))\dot{\phi}
                       =\rho\circ((\alpha\circ\phi)\dot{\phi})$
where the second equality follows by 2.\ above.
To check the second equality needed for $(\alpha\circ\phi)\dot{\phi}$
to be an ${\mathcal A}$-lift of $c\circ\phi$
one just has to note that
$\pi((\alpha\circ\phi)\dot{\phi})=\pi(\alpha\circ\phi)=c\circ\phi$
where the first and last equality follow by 1.\ above, while the
second equality follows by the fact that $\pi(tv)=\pi(v)$ for any real
$t$ and any $v\in{\mathcal A}$.
\end{proof}

\begin{lemma}\label{smooth-lift}
	For each $y$ in 
	the $\mathcal A$-orbit of $x$ there exists an $\mathcal A$-admissible smooth curve $\gamma\colon[a,b]\to M$ with a smooth $\mathcal A$-lift such that $\gamma(a)=x$ and $\gamma(b)=y$.
\end{lemma}

\begin{proof}
By the definition of an $\mathcal A$-orbit there exists a piecewise smooth curve $c\colon[a,b]\to\mathcal A$ such that $c(a)=x$ and $c(b)=y$ with a piecewise smooth $\mathcal A$-lift~$\alpha$.
There exists a smooth bijection $\phi\colon[a,b]\to[a,b]$ for which
both $c\circ\phi$ and $\alpha\circ\phi$ are smooth. 
As $\phi$ one may take any smooth homeomorphism whose derivatives of
arbitrary order vanish at the finitely many points where $c$ or $\alpha$ fail to be smooth. 
Then from Lemma \ref{reparam} it follows that $\gamma = c\circ\phi$ is the required smooth curve with its smooth $\mathcal A$-lift $(\alpha\circ\phi)\dot{\phi}$.
\end{proof}

Lemma~\ref{smooth-lift} shows that the algebroid orbits introduced in Definition \ref{LBA-orbit} can  be equivalently defined using only smooth paths that are $\mathcal A$-admissible 
in the sense defined in \cite{CraFe03} for the case of finite-dimensional Lie algebroids.

\begin{definition}\label{split}
\normalfont  
A Banach-Lie algebroid $({\mathcal{A}},{M},\rho,[\cdot,\cdot]_{\mathcal{A}})$ is called 
\emph{split on a connected component}  ${N}$ of ${M}$ if for each $x\in {N}$ the kernel of $\rho_x=\rho\vert_{ \pi^{-1}(x)}$ is 
a split subspace of $ \pi^{-1}(x)$. 
The algebroid is called \emph{split} if it is split on all the connected components of ${M}$.
\end{definition}

 If $({\mathcal{A}},{M},\rho,[\cdot,\cdot]_{\mathcal{A}})$ is a finite-dimensional Lie algebroid,  it is well known that each ${\mathcal{A}}$-orbit is a smooth immersed manifold of $M$. In the Banach context from \cite[Th. 5]{Pe} applied to each subbundle $\pi:{\mathcal{A}}_\alpha\ap \pi({\mathcal{A}}_\alpha)$ we draw the following conclusion. 

\begin{theorem}\label{orbitalgebroid} 
If $({\mathcal{A}},{M},\rho,[\cdot,\cdot]_{\mathcal{A}})$ is a split Banach-Lie algebroid and the distribution $\rho(\mathcal{A})$  on $M$ is closed, then each orbit is 
an immersed   
submanifold  of $M$. 
\end{theorem}

\subsection{The Banach-Lie algebroid of a n.n.H. Banach-Lie groupoid}

The main result of this subsection is Theorem~\ref{algbroidgroupoid} on 
the construction of the Banach-Lie algebroid of a n.n.H. Banach-Lie   groupoid $\Gc\tto M$. 
This is formally the same as in the case of 
finite-dimensional Lie groupoids (cf. \cite{Ma}) 
and, for infinite-dimensional Lie groupoids modelled on locally convex spaces it was considered in \cite{SW15} and \cite{SW16}. 
However,  since the basis~$M$ may not be smoothly regular  
(see Remark~\ref{smooth_reg}), 
the localizability property of the Lie bracket must be proved separately. 
This is a technical aspect that occurs neither in the case of finite-dimensional Lie groupoids nor in the case of Banach-Lie groups. 
In order to deal with this problem we use the following notion.

\begin{definition}\label{right_act}
	\normalfont
	Let $\Gc\tto M$ be a n.n.H. Banach-Lie groupoid, 
	$H$ be a n.n.H. Banach manifold, and $\sigma\colon H\to M$ be a surjective submersion. 
	Then the set 
	$H\ast \Gc:=\{(h,g)\in H\times\Gc\mid\sigma(h)={\bf t}(g)\}$ 
	is the inverse image of the diagonal of $M\times M$ though the submersion 
	$(\sigma,{\bf t})\colon H\times\Gc\to M\times M$, hence 
	$H\ast \Gc$ is a closed split submanifold of $H\times\Gc$. 
	
	In the above framework,  a \emph{right action of $\Gc$ on $H$} is a smooth map $H\ast \Gc\to H$, $(h,g)\mapsto h.g$, satisfying the following conditions: 
	\begin{itemize}
		\item If $g_1,g_2\in\Gc$ with ${\bf s}(g_1)={\bf t}(g_2)$, and $h\in H$ with $\sigma(h)={\bf t}(g_1)$, 
		then $h.(g_1g_2)=(h.g_1).g_2$. 
		\item If $(h,g)\in H\ast\Gc$ then $\sigma(h.g)=\sigma(h)$. 
	\end{itemize}
	
	We denote $T^\sigma H:=\Ker(T\sigma)\subseteq TH$, hence $T^\sigma H$ is an integrable distribution on~$H$, 
	since $\sigma\colon H\to M$ is a submersion.  
	We denote by $\Gamma(T^\sigma H)$ the vector space of smooth sections of $T^\sigma H$ 
	regarded as a sub-bundle of the tangent bundle $TH\to H$. 
	
	For every $g\in\Gc$ one has the diffeomorphism 
	\begin{equation}\label{diffeo_R}
	R_g\colon \sigma^{-1}({\bf t}(g))\to\sigma^{-1}({\bf s}(g)),\quad h\mapsto h.g.
	\end{equation}
	We denote by $\Gamma_{\rm inv}(H)$ the set of all vector fields $X\in \Gamma(T^\sigma H)$ 
	satisfying the invariance condition 
	\begin{equation}\label{invar_R}
	X_{h.g}=(T_h(R_g))(X_h)\text{ for all }g\in\Gc\text{ and all } h\in\sigma^{-1}({\bf t}(g)). 
	\end{equation}
\end{definition}

%

\begin{remark}
	\normalfont
	Assume the setting of Definition~\ref{right_act}. 
	Since $T^\sigma H$ is an integrable distribution, 
	it follows that $\Gamma(T^\sigma H)$ is a Lie algebra of smooth vector fields on~$H$. 
	We also note that if $X\in\Gamma(T^\sigma H)$ and $h\in H$, 
	then $\sigma^{-1}(\sigma(h))\subseteq H$ is a closed submanifold and one has 
	$X_h\in T_h(\sigma^{-1}(\sigma(h)))\subseteq T_h H$. 
\end{remark}

We prove the following lemma for the sake of completeness 
---its proof is similar to that of its counterpart from the construction of the Lie algebroid of a finite-dimensional Lie groupoid \cite{Ma05}. 

\begin{lemma}\label{3_right} 
	In the setting of Definition~\ref{right_act}, 
	the following assertions hold. 
	\begin{enumerate}[{\rm(i)}]
		\item\label{3_right_item1} 
		Let $U\subseteq H$ be a subset such that  
		$H=\{u.g\mid (u,g)\in H\ast\Gc\text{ and }u\in U\}$. 
		If $X$ is a smooth vector field on $H$, 
		then one has 
		$X\in \Gamma_{\rm inv}(H)$ if and only if 
		$X_u\in T_u(\sigma^{-1}(\sigma(u)))$ and $X_{u.g}=TR_g(X_u)$ for all $(u,g)\in H\ast\Gc$ with $u\in U$.  
		In particular, the value of a right invariant vector field is determined by its values at the points of~$U$.
		\item\label{3_right_item2} 
		If $X,Y\in\Gamma_{\rm inv}(H)$, then $[X,Y]\in\Gamma_{\rm inv}(H)$.
		\item\label{3_right_item3} 
		$\Gamma_{\rm inv}(H)$ is a subalgebra of the Lie algebra $\Gamma(T^\sigma H)$. 
	\end{enumerate}
\end{lemma}

\begin{proof} 
	If $X\in \Gamma_{\rm inv}(H)$ then it clearly satisfies the conditions from the statement. 
	Conversely, let us assume that $X$ is a smooth vector field  satisfying these conditions. 
	For arbitrary $h\in H$ there exists $(u,g)\in H\ast\Gc$ with $u\in U$ and $h=u.g$. 
	Then $X_h=(T_u(R_g))(X_u)$ with $X_u\in T_u(\sigma^{-1}(\sigma(u)))=T_u(\sigma^{-1}({\bf t}(g)))$. 
	Using the diffeomorphism~\eqref{diffeo_R} and the fact that $R_g(u)=u.g=h$, we then obtain 
	$X_h\in T_h(\sigma^{-1}({\bf s}(g)))\subseteq T^\sigma H$. 
	Thus $X\in\Gamma(T^\sigma H)$. 
	It is straightforward to check that $X$ also verifies the invariance condition~\eqref{invar_R}, 
	hence $X\in\Gamma_{\rm inv}(H)$.

	Now assume that $X,Y\in\Gamma_{\rm inv}(H)$. 
	In particular $X,Y\in\Gamma(T^\sigma H)$ hence, 
	since the distribution $T^\sigma H\subseteq TH$ is integrable, 
	we obtain $[X,Y]\in\Gamma(T^\sigma H)$. 
	For any $g\in \Gc$, it follows by \eqref{invar_R} 
	that the restrictions of $X$ to the manifolds $\sigma^{-1}({\bf s}(g))$ and $\sigma^{-1}({\bf t}(g))$ are $R_g$-related, using the diffeomorphism~\eqref{diffeo_R}. 
	Similarly, the restrictions  of $Y$ to the manifolds $\sigma^{-1}({\bf s}(g))$ and $\sigma^{-1}({\bf t}(g))$ are $R_g$-related. 
	Since both $X$ and $Y$ are tangent to the manifolds $\sigma^{-1}({\bf s}(g))$ and $\sigma^{-1}({\bf t}(g))$, 
	it then follows that 
	the restrictions  of $[X,Y]$ to the manifolds $\sigma^{-1}({\bf s}(g))$ and $\sigma^{-1}({\bf t}(g))$ are $R_g$-related. 
	This is equivalent to the fact that $[X,Y]$ satisfies the invariance condition~\eqref{invar_R} with $X$ replaced by $[X,Y]$. 
	Thus $[X,Y]\in\Gamma_{\rm inv}(H)$. 
	
	The third assertion in the statement follows by the second assertion, and this concludes the proof. 
\end{proof}

\begin{proposition}\label{invarianttarget_right} 
	Let $\Gc\tto M$ be a n.n.H. Banach-Lie groupoid, 
	$U\subseteq M$ be a nonempty open subset, 
	and define $\Gc^U:={\bf t}^{-1}(U)\subseteq\Gc$.

	Then the  following assertions hold.  
	\begin{enumerate}[(i)]
		\item\label{invarianttarget_right_item1}  
		For  $H:=\Gc^U$ and $\sigma:={\bf s}\vert_{\Gc^U}\colon \Gc^U\to M$, the groupoid multiplication defines by restriction 
		a right action $H\ast\Gc\to H$. 
		\item\label{invarianttarget_right_item2}  
		If  $X \in  \Gamma_{\rm inv}(\Gc^U)$, then 
		$T{\bf t}(X_g)=T{\bf t}(X_{{\bf 1}_{{\bf t}(g)}})$
		for all $g\in \mathcal{G}$.
		\item\label{invarianttarget_right_item3}   
		$T{\bf t}$ induces a morphism of Lie algebras from $\Gamma_{\rm inv}(\Gc^U)$ into 
		the Lie algebra of vector fields on $U$. 
	\end{enumerate}
\end{proposition}

\begin{proof} For Assertion~\eqref{invarianttarget_right_item1} we only need that ${\bf s},{\bf t}\colon\Gc\to M$ are surjective submersions. 
	
	For Assertion~\eqref{invarianttarget_right_item2}  consider $g\in \Gc^U$ and set $x={\bf s}(g)$ and $y={\bf t}(g)$. Then we have:
	$$T{\bf t}(X_g)=T{\bf t}(X_{{\bf 1}_yg})=T{\bf t}\circ TR_g(X_{{\bf 1}_y})=T({\bf t}\circ R_g)(X_{{\bf 1}_y}).$$
	Since ${\bf t}_x\circ R_g={\bf t}_y$
	we get $T{\bf t}(X_g)=T{\bf t}(X_{{\bf 1}_y})$. 
	
	Assertion~\eqref{invarianttarget_right_item3} follows by 
	Assertions \eqref{invarianttarget_right_item1}--\eqref{invarianttarget_right_item2} 
	along with Lemma~\ref{3_right}.
\end{proof}

For any nonempty open subset $U\subseteq M$ 
we  now define the Banach vector bundle $\pi_U\colon (\mathcal{AG})_U\to U $  as the pullback of 
the Banach bundle $T^{\bf s}(\Gc^U)\to\Gc^U$ 
through the map ${\bf 1}\colon U\to\Gc^U$.   
(If we identify $U$ with its image through the map  ${\bf 1}\colon U\to\Gc$, then we may say that $(\mathcal{AG})_U$ is the 
restriction of  $T^{\bf s}(\Gc^U)$ to $U$.) 
Note that $(\mathcal{AG})_U$ is a Banach manifold since $M$ is Hausdorff. 

One has a natural Banach bundle morphism $(\1_U)_*\colon(\mathcal{AG})_U\to T^{\bf s}(\Gc^U) $ 
over the map ${\bf 1}\colon U\to\Gc^U$, from the bundle  $(\mathcal{AG}_U)\to U$ to the bundle $T^{\bf s}(\Gc^U)\to \Gc^U$,  which is a bundle isomorphism onto the restriction of the bundle $T^{\bf s}(\Gc^U)\to\Gc^U$ to $\1(U)\subseteq\Gc^U$. 
On the other hand, let  $\Gamma_{\1}(T^{\bf s}(\Gc^U))$ be the vector space of all smooth sections of  the restriction of the bundle $T^{\bf s}(\Gc^U)\to\Gc^U$ to $\1(U)\subseteq\Gc^U$. 
By Lemma~\ref{3_right} we obtain a Lie bracket on $\Gamma_{\bf 1}(T^{\bf s}(\Gc^U))$ with a Lie algebra isomorphism   
$\Gamma_{\bf 1}(T^{\bf s}(\Gc^U))\to\Gamma_{\rm inv}(\Gc^U)$. 
Now the aforementioned bundle ismorphism  $(\1_U)_*$ gives rise to 
a linear isomorphism $\Gamma((\mathcal{AG})_U)\to\Gamma_{\bf 1}(T^{\bf s}(\Gc^U))$. 
We thus obtain a linear isomorphism 
\begin{equation}\label{isom_right}
\Gamma((\mathcal{AG})_U)\to \Gamma_{\rm inv}(\Gc^U),\quad 
X\mapsto\widetilde{X}
\end{equation}
and we thus see that there exists a unique Lie bracket $[\cdot,\cdot]_U$ on 
the vector space $\Gamma((\mathcal{AG})_U)$ for which the above  linear isomorphism $X\mapsto\widetilde{X}$ is a Lie algebra isomorphism. 
For later reference we note that 
\begin{equation}\label{bracket_right}
\widetilde{[X,Y]}=[\widetilde{X},\widetilde{Y}]
\text{ for all }X,Y\in\Gamma((\mathcal{AG})_U)
\end{equation}
and 
\begin{equation}\label{fY_right}
\widetilde{fX}=(f\circ{\bf t})\tilde{X} 
\text{ for all }X\in\Gamma((\mathcal{AG})_U)\text{ and }f\in C^\infty(M).
\end{equation} 

\begin{proposition}\label{algebroid_right}
	For every open subset $U\subseteq M$, the map 
	$$\rho_U:=T{\bf t}\circ (\1_U)_*\colon (\mathcal{AG})_U\to TU$$ is a bundle morphism over $\id_{U}$ 
	whose corresponding map  $\rho_U\colon\Gamma((\mathcal{AG})_U)\to\Gamma(TU)$  is a morphism of Lie algebras. 
	Moreover for any $f\in C^\infty(M)$ and $X, Y\in \Gamma((\mathcal{AG})_U)$ we have the Leibniz property:
	$$[X,fY]_U=f[X,Y]_U+\de f(\rho_U(X))Y.$$
\end{proposition}

\begin{proof}  The definition of $\rho_U$  implies clearly that it is a bundle morphism.  From the construction of the isomorphism~\eqref{isom_right} and  
	Proposition \ref{invarianttarget_right}\eqref{invarianttarget_right_item3} it follows that $\rho_U$ induces a morphism of Lie algebras as indicated in the statement.  
	Using \eqref{bracket_right} and \eqref{fY_right} we have 
	\allowdisplaybreaks
	\begin{align}
	\widetilde{[X,fY]}
	&=[\widetilde{X},\widetilde{fY}] \nonumber\\
	&=[\widetilde{X},(f\circ{\bf t})\widetilde{Y}] \nonumber\\
	&=(f\circ{\bf t})[\widetilde{X},\widetilde{Y}]+\de (f\circ{\bf t})(\widetilde{X})\widetilde{Y} \nonumber\\
	&=(f\circ{\bf t})\widetilde{[X,Y]}+\de f(T{\bf t}(\widetilde{X}))\widetilde{Y} \nonumber\\
	&=\widetilde{(f\circ{\bf t})[X,Y]}+\widetilde{\de f(\rho(X))Y}, 
	\nonumber
	\end{align} 
	(see \cite[page 152]{MoMr}). 
	Using the isomorphism~\eqref{isom_right} again, we are done.
\end{proof}

\begin{notation}
	\normalfont
	In Proposition~\ref{invarianttarget_right}, if $U=M$ then $\Gc^U=\Gc$. 
	Therefore, in the above constructions we drop $U$ from the notation whenever $U=M$. 
	We thus denote $\Ac\Gc$ instead of $(\Ac\Gc)_M$, 
	and we obtain a Lie bracket $[\cdot,\cdot]$ (rather than $[\cdot,\cdot]_M$) on $\Gamma(\Ac\Gc)$. 
	Also, we use the notation $\rho:=\rho_M\colon\Ac\Gc\to TM$ 
	for the map defined in Proposition~\ref{algebroid_right}. 
\end{notation}

We are now in a position to prove the following theorem, 
which lays the foundations of a Lie theory for n.n.H. Banach-Lie groupoids. 

\begin{theorem}\label{algbroidgroupoid}
	For every n.n.H. Banach-Lie groupoid  $\Gc\tto M$, 
	the structure $(\mathcal{AG},M,\rho,[\cdot,\cdot])$ is a Banach-Lie algebroid.
\end{theorem}

\begin{proof}
	It follows by Proposition~\ref{algebroid_right} that for every open subset $U\subseteq M$ one has a Lie bracket $[\cdot,\cdot]_U$ on the Banach anchored bundle $((\Ac\Gc)_U,\pi_U,U,\rho_U)$. 
	It follows directly by the construction that $\pi_U\colon (\Ac\Gc)_U\to U$ is the restriction of the bundle $\pi=\pi_M\colon \Ac\Gc\to M$ to~$M$. 
	
	We still need to prove that the Lie bracket $[\cdot,\cdot]$ on $\Gamma(\Ac\Gc)$ is localizable, and to this end we must check that for any open subsets $V\subseteq U\subseteq M$ the following conditions are satisfied: 
	\begin{itemize}
		\item 
		The vector bundle $(\Ac\Gc)_V$ is the restriction of the vector bundle $(\Ac\Gc)_U$ to~$V$. 
		\item The restriction map 
		$\Gamma((\Ac\Gc)_U)\to\Gamma((\Ac\Gc)_V)$, 
		$\sigma\mapsto\sigma\vert_V$, is a morphism of Lie algebras with respect to the Lie brackets $[\cdot,\cdot]_U$ and $[\cdot,\cdot]_V$ on 	$\Gamma((\Ac\Gc)_U)$ and $\Gamma((\Ac\Gc)_V)$, respectively. 
	\end{itemize}
	For the first of these conditions we note that $\Gc^V$ is an open subset of $\Gc^U$, and the vector bundle $T^{\bf s}(\Gc^V)\to\Gc^V$ is equal to the vector bundle $T^{\bf s}(\Gc^U)\to\Gc^U$ 
	restricted to $\Gc^V$. 
	One then uses the definition of the vector bundles $(\Ac\Gc)_V$ and $(\Ac\Gc)_U$. 
	For the second of the above conditions we note the commutative diagram 
	$$\begin{tikzcd}[column sep=large]
	\Gamma((\Ac\Gc)_U) \arrow{r}{X\mapsto\widetilde{X}} \arrow{d} & 
	\Gamma_{\rm inv}(\Gc^U)  \arrow[r, hook] \arrow{d} & 
	\Xi(\Gc^U) \arrow{d}\\
	\Gamma((\Ac\Gc)_V) \arrow{r}{X\mapsto\widetilde{X}} & 
	\Gamma_{\rm inv}(\Gc^U) \arrow[r, hook] & 
	\Xi(\Gc^V)
	\end{tikzcd}$$
	where we denoted by $\Xi(N)$ the Lie algebra of all smooth vector fields on any n.n.H. Banach manifold~$N$, 
	the vertical arrows are restriction maps, 
	and the maps $X\mapsto\widetilde{X}$ are Lie algebra isomorphisms as in~\eqref{isom_right}. 
	Since  $\Gc^V$ is an open subset of $\Gc^U$, it is well known that the restriction map $\Xi(\Gc^U)\to\Xi(\Gc^V)$ is a Lie algebra morphism, and then by the above commutative diagram we obtain that the restriction map $\Gamma((\Ac\Gc)_U)\to\Gamma((\Ac\Gc)_V)$ is a Lie algebra morphism, too. 
	This completes the proof. 
\end{proof}


\begin{remark}\label{summary}
\normalfont
Given a n.n.H. Banach-Lie groupoid $\Gc\tto M$, 
the construction of the Banach-Lie algebroid $(\mathcal{AG},M,\rho,[\cdot,\cdot])$ in Theorem~\ref{algbroidgroupoid} can be summarized as follows: 
\begin{itemize}
	\item The Banach vector bundle $\pi\colon \Ac\Gc\to M$ has its fiber $(\Ac\Gc)_x=T_{\1_x}(\Gc(x,-))$ at any $x\in M$. 
	\item The anchor $\rho\colon \Ac\Gc\to TM$ 
	is a morphism of vector bundles 
	whose fiber at $x\in M$ is the differential of the target map 
	${\bf t}_x={\bf t}\vert_{\Gc(x,-)}\colon \Gc(x,-)\to M$ at the point $\1_x\in\Gc(x,-)$. 
\end{itemize}
Hence, just as in the special case of Banach-Lie groups, the underlying structure of the anchored bundle of a Banach-Lie algebroid of a n.n.H. Banach-Lie groupoid 
 does not require invariant vector fields. 
These vector fields are only needed in order to define the Lie bracket $[\cdot,\cdot]$ on $\Gamma(\Ac\Gc)$ for which the anchor $\rho$ is a morphism of Lie algebras, which extends one of the methods to define the Lie bracket on the Lie algebra of a Lie group.  
We note however that there is no essential difference between the total space $\Ac\Gc$ and the section space $\Gamma(\Ac\Gc)$ in the case of Lie groups, where the base of the vector bundle $\Ac\Gc$ is a singleton $M=\{\1\}$. 
See also Remark~\ref{functors} below. 
\end{remark}

Using the above remark we now make the following definition. 

\begin{definition}\label{Lie_functor}
\normalfont 
If $\Gc\tto M$ is a n.n.H. Banach-Lie groupoid, then 
the Banach-Lie algebroid $(\mathcal{AG},M,\rho,[\cdot,\cdot])$ constructed in Theorem~\ref{algbroidgroupoid} is said to be \emph{the Banach-Lie algebroid associated to the n.n.H. Banach-Lie groupoid  $\mathcal{G}\tto M$}.  
For the  sake of simplicity  this algebroid will be {\bf simply denoted $\mathcal{AG}$}. 

Let $\Hc\tto M$ be another n.n.H. Banach-Lie groupoid 
and $\Phi\colon\Gc\to\Hc$ be a morphism of Banach-Lie groupoids over $\id_M$. 
Then ${\bf s}\circ \Phi={\bf s}$, 
hence for every $x\in M$ one has $\Phi(\Gc(x,-))\subseteq\Hc(x,-)$, and we may define the bounded linear operator 
$$(\Ac\Phi)_x\colon (\Ac\Gc)_x\to(\Ac\Hc)_x,\quad 
(\Ac\Phi)_x:=(T_{\1_x}\Phi)\vert_{(\Ac\Gc)_x}\colon  (\Ac\Gc)_x\to (\Ac\Hc)_x$$
We then define $\Ac\Phi\colon\Ac\Gc\to\Ac\Hc$ as the morphism of Banach vector bundles over $\id_M$ whose restriction to the fiber $(\Ac\Gc)_x$ is the above operator $(\Ac\Phi)_x$ for every $x\in M$. 
\end{definition}

\begin{remark}\label{functors}
	\normalfont
In Definition~\ref{Lie_functor}, for every $x\in M$ one has the commutative diagram 
\[\begin{tikzcd}[column sep=large]
\Gc(x,-) \arrow{r}{\Phi\vert_{\Gc(x,-)}} \arrow{d}[swap]{{\bf t}}
&\Hc(x,-) \arrow[d,"{\bf t}"] \\
M  \arrow[r, "\id_M"] 
&M
\end{tikzcd}
\]
from which, computing the tangent maps at $\1_x\in\Gc(x,-)$, 
one obtains 
\[\begin{tikzcd}[column sep=large]
(\Ac\Gc)_x \arrow{r}{(\Ac\Phi)_x} \arrow{d}[swap]{\rho_x}
&(\Ac\Hc)_x \arrow[d,"\rho_x"] \\
T_xM  \arrow[r, "\id_{T_xM}"] 
& T_xM
\end{tikzcd}
\]
which is again a commutative diagram, 
where we denoted by $\rho$ the anchor maps of both Banach-Lie algebroids $\Ac\Gc\to M$ and $\Ac\Hc\to M$. 

In addition to the above, one can check that $\Ac\Phi\colon\Ac\Gc\to\Ac\Hc$ is a morphism of Banach-Lie algebroids over $\id_M$, 
just as in finite dimensions. 
(See for instance \cite[Prop. 3.5.10]{Ma05}.)
Moreover, let $\mathbb{GRPD}_M$ be  the category of n.n.H. Banach-Lie groupoids with the base $M$, in which the morphisms are the morphisms of Banach-Lie groupoids over $\id_M$. 
Also let  $\mathbb{ALGBD}_M$ be  the category of Banach-Lie algebroids with the base $M$, in which the morphisms are the morphisms of Banach-Lie algebroids over $\id_M$. 
Then one can also check that 
$\Ac\colon \mathbb{GRPD}_M\to\mathbb{ALGBD}_M$ 
is a functor, which is equal to the Lie functor from Banach-Lie groups to Banach-Lie algebras in the special case when $M=\{\1\}$ is a singleton, with the Lie bracket defined via right-invariant vector fields on Banach-Lie groups. 
(See \cite[\S 4]{SW15} for a more complete discussion.)
\end{remark}

In the sequel, we will need the following result:

\begin{proposition}\label{kerrho} 
	For any $y\in M$ the kernel of the restriction of $\rho$ to the fiber $\mathcal{AG}_y=\pi^{-1}(y)$ is exactly ${\bf 1}_*^{-1}(\ker T_{{\bf 1}_y}{\bf s}\cap \ker T_{{\bf 1}_y}{\bf t})$. 
In particular $\mathcal{AG}$ is split on the connected component ${N}$ of $M$ if and only if 
$\ker T_{{\bf 1}_y}{\bf s}\cap \ker T_{{\bf 1}_y}{\bf t}$ is 
a split subspace of 
$\ker T_{{\bf 1}_y}{\bf s}$ for all $y\in {N}$. In particular $\mathcal{G}$ is split if and only if $\mathcal{AG}$ is split.
\end{proposition}

\begin{proof} 
Recall that  by construction $\rho$ is the composition of $T{\bf t}$  restricted to  $\ker T_{{\bf 1}_y}{\bf s}$ with ${\bf 1}_*$. Since ${\bf 1}_*$ is an isomorphism in restriction to $\mathcal{AG}_y$ we obtain directly the first part. 
The second part is a direct consequence of Definition \ref{split} and the previous argument. The last assertion is  then a direct consequence of Theorem  \ref{5.4}\eqref{smooth4_item1} since ${\bf s}$ is a submersion.
\end{proof}

\subsection{Link between Banach-Lie  algebroids  and  n.n.H. Banach-Lie groupoid orbits}\label{AL-G-orbits}

 

 \begin{theorem}\label{linkorbitG}
If $\Gc\tto M$ is a split n.n.H. Banach-Lie groupoid,   
then for any $x\in M$  
its orbit  $\Gc.x$ is a weakly immersed submanifold of $M$ whose connected components are orbits  of the  Banach-Lie algebroid $\mathcal{AG}$.
 \end{theorem}
 
  \begin{remark}\label{splitG}
	\normalfont
\begin{enumerate}
 \item
According to Theorem \ref{5.4}, 
the previous Theorem~\ref{linkorbitG} implies that the orbits of an integrable split Banach-Lie algebroid (cf. section \ref{tildeG}) 
are always 
weakly immersed Banach submanifold of $M$. 
On the other hand for a general Banach-Lie algebroid $(\mathcal{A},M,\rho, [\cdot,\cdot])$ if this algebroid is split  and the associated distribution $\rho(\mathcal{A})$ is closed,   then Theorem \ref{orbitalgebroid} 
shows that its orbits are immersed submanifolds. 
In finite dimensions,  any orbit of a Lie algebroid is a 
split immersed submanifold of $M$. 
In the general context of Banach-Lie algebroids 
we think that the orbits may not be weakly immersed submanifolds but unfortunately, we have no specific example in this connection.
\item
The Banach-Lie algebroid of any split n.n.H. Banach-Lie groupoid is split. 
Indeed  for any $x\in M$, we have  $\ker\rho_x=T_x(\Gc(x))$, which is the Lie algebra of the isotropy group $\Gc(x)$.
As in the proof of Theorem~\ref{5.4}\eqref{smooth4_item2}, for  the sequence 
$\Gc(x)\subseteq\Gc(x,-)\subseteq\Gc$  we know that $\Gc(x)\subseteq\Gc$ and $\Gc(x,-)\subseteq\Gc$ are submanifolds, 
and so by Lemma~\ref{smooth3} also $\Gc(x)\subseteq\Gc(x,-)$ is a submanifold. 
Therefore  
the inclusion of tangent spaces $\Ker\rho_x\subseteq (\Ac\Gc)_x$ is split 
since $(\Ac\Gc)_x=T_x(\Gc(x,-))$ and ${\bf t}_x:\Gc(x,-)\ap \Gc.x$ is a submersion. 
\end{enumerate}
 \end{remark}

The proof of Theorem \ref{linkorbitG}  is based on Lemma~\ref{linkorbit} below, which in turn needs some preparations.

For any connected component ${N}$ of $M$  
we denote 
$$\Gc_N:=\{g\in \mathcal{G}\mid {\bf s}(g) \in {N},  {\bf t}(g) \in {N}\}$$
and let $\mathcal{G}_{{N}}^0$ 
be the connected component of $\mathcal{G}_{{N}}$ that contains ${\bf 1}({N})$. 
Just as for finite-dimensional Lie groupoids 
(see for instance \cite[Prop. 1.30]{CraFe}), 
it is easy to show that  $\mathcal{G}_{{N}}\tto {N}$ (resp. $\mathcal{G}_{{N}}^0\tto {N}$) 
provided with the restrictions  ${\bf s}_{{N}}$, ${\bf t}_{{N}}$ and ${\bf i}_{{N}}$ 
(resp. ${\bf s}_{{N}}^0$, ${\bf t}_{{N}}^0$ and ${\bf i}_{{N}}^0$)
  of ${\bf s}$, ${\bf t}$ and $\iota$ to  $\mathcal{G}_{{N}}$ (resp. $\mathcal{G}_{{N}}^0$), 
	and the restriction  ${\bf m}_{{N}}$ (resp. ${\bf m}_{{N}}^0$) of  ${\bf m}$ to $\mathcal{G}_{{N}}^{(2)}$ (resp. $(\mathcal{G}_{{N}}^0)^{(2)}$) is a n.n.H.  Banach-Lie groupoid  (resp.  ${\bf s}$-connected n.n.H. Banach-Lie groupoid). 
Let $\mathcal{AG}_{{N}}\ap {N}$ be  the restriction of $\mathcal{AG}$ to ${N}$. 
	Then the restriction $\rho_{{N}}$ of $\rho$ to $\mathcal{AG}_{{N}}$ take values in $T{N}$ and the restriction of the Lie bracket $[\cdot,\cdot]$ on $\mathcal{AG}$  to $\mathcal{AG}_{{N}}$ is a Lie bracket on the anchored bundle $(\mathcal{AG}_{{N}},{N},\rho_{{N}})$ is a Lie bracket again denoted by $[\cdot,\cdot]$. 
	Finally  $(\mathcal{AG}_{{N}},{N},\rho_{{N}},[\cdot,\cdot])$  (denoted as $\mathcal{AG}_{{N}}$ for short) is the algebroid associated to $\mathcal{G}_{{N}}\tto  {N}$.
Note that $\mathcal{AG}_{{N}}$ is also the Banach-Lie algebroid of $\mathcal{G}_{{N}}^0\tto  {N}$ 
(cf. Remark~\ref{summary}).

We define $\Gc^0$ as the set of all $g\in\Gc$ that belong to the 
connected component of $\1_{{\bf s}(g)}$ in $\Gc({\bf s}(g),-)$, 
and then $\Gc^0$ is an ${\bf s}$-connected open wide subgroupoid of $\Gc$ just as in \cite[Prop. 1.30]{CraFe}.

\begin{lemma}\label{linkorbit} 
Let ${N}$ be a connected component of $M$. For any $x\in N$,
its  $\mathcal{G}_{{N}}^0$-orbit 
$\mathcal{G}_{{N}}^0.x$ 
coincides with the $\mathcal{AG}_{{N}}$-orbit of $x$ 
and is a weakly immersed submanifold of~$N$.  
  \end{lemma}

We postpone the proof of Lemma~\ref{linkorbit} until after the proof of Theorem~\ref{linkorbitG}. 

\begin{lemma}\label{linkorbit1} 
Let $\Oc\subseteq M $ be an orbit of $\Gc$, 
and assume that $\Oc$ is endowed with a topology such that for every $x\in\Oc$ the map 
${\bf t}\vert_{\Gc(x,-)}\colon\Gc(x,-)\to\Oc$ is continuous and open. 
Then $\{\Gc^0.x\mid x\in\Oc\}$ is the family of all connected components of~$\Oc$. 
\end{lemma}

\begin{proof}
It suffices to prove that $\{\Gc^0.x\mid x\in\Oc\}$ is a family of open, connected, mutually disjoint subsets of $\Oc$  whose union is equal to~$\Oc$. 

To this end we first note that for arbitrary $x\in\Oc$ one has $x\in \Gc^0.x$, 
hence $\bigcup\limits_{x\in\Oc}\Gc^0.x=\Oc$. 
We now check that if $x,y\in\Oc$ and $\Gc^0.x\cap\Gc^0.y\ne\emptyset$ then 
$\Gc^0.x=\Gc^0.y$. 
In fact, if $z\in\Gc^0.x\cap\Gc^0.y$, 
then there exist $g\in\Gc^0(x,z)$ and $h\in\Gc^0(y,z)$.  
For every $w\in\Gc^0.x$ there exists $k\in\Gc^0(x,w)$, 
and then $kg^{-1}h\in\Gc^0(y,w)$, hence $w\in\Gc^0.x$.  
This shows that $\Gc^0.x\subseteq\Gc^0.y$, and the converse inclusion can be proved similarly. 

We have recalled above that $\Gc^0$ is an ${\bf s}$-connected groupoid. 
Since the image of a connected set by a continuous map is connected  and one has 
$\Gc^0.x={\bf t}^0(\Gc^0(x,-))$, 
it then follows that $\Gc^0.x$ is a connected subset of $\Oc$. 
Moreover, as also recalled above, $\Gc^0$ is an open subset of $\Gc$, 
and this implies that $\Gc^0(x,-)=\Gc^0\cap\Gc(x,-)$ is an open subset of $\Gc(x,-)$. 
The map  
${\bf t}\vert_{\Gc(x,-)}\colon\Gc(x,-)\to\Oc$ is open by hypothesis, 
hence the equality $\Gc^0.x={\bf t}^0(\Gc^0(x,-))$ implies that $\Gc^0.x$ is an open subset of $\Oc$, and this completes the proof. 
\end{proof}

 \begin{proof}[Proof of Theorem \ref{linkorbitG}] 
 	 By Theorem \ref{5.4}\eqref{smooth4_item2}, $\mathcal{G}(x,-)$ is a pure Hausdorff  Banach manifold 
 	and the map ${\bf t}_x: \mathcal{G}(x,-)\ap \mathcal{G}.x$ is a locally trivial fibration, hence is in particular continuous and open.  
 The assertion then follows at once by Lemmas~\ref{linkorbit} and \ref{linkorbit1}. 
 \end{proof}

  \begin{proof}[Proof of Lemma \ref{linkorbit}] 
  
{\bf Step 1}: 
The  $\mathcal{G}_{{N}}^0$-orbit of $x$ is contained in its $\mathcal{AG}_{{N}}$-orbit.

   For every $y\in \mathcal{G}_{{N}}^0.x$ 
   there exists $g\in \mathcal{G}_{{N}}^0$ with ${\bf s}(g)=x$ and  ${\bf t}(g)=y$. 
	Since the groupoid $\mathcal{G}_{{N}}^0\tto  {N}$ is ${\bf s}$-connected, there exists a  smooth curve $\gamma:[0,1]\ap \mathcal{G}_{{N}}^0(y,-)$ with $\gamma(0)={\bf 1}_x$ and $\gamma(1)=g$. 
	For  $c:={\bf t}\circ \gamma\colon [0,1]\ap {N}$ 
	one has 
	$c(0)=x$ and $c(1)=y$. 
	Hence in order to prove that $y$ belongs to the $\mathcal{AG}_{{N}}$-orbit of $x$ it suffices to show that $c$ is an $\Ac\Gc_N^0$-admissible curve. 
		To this end 
	we set  
	$$\alpha(t)=T R_{\gamma(t)^{-1}}(\dot{\gamma}(t))
	\in T_{\1_{{\bf t}(\gamma(t))}}(\Gc_N^0({\bf t}(\gamma(t)),-)) \text{ for every }t\in[0,1]$$ 
	and we will check that $\alpha$ is an $\Ac\Gc_N^0$-lift of~$c$.  
	
	We recall from Remark~\ref{summary} that the fiber of $\Ac\Gc_N^0$ at any 
	$z\in N$ is $T_{\1_{z}}(\Gc_N^0(z,-))$. 
	Hence, if we denote by $\pi\colon \Ac\Gc_N^0\to N$ the vector bundle projection, then one has 
	$\pi(\alpha(t))={\bf t}(\gamma(t))=c(t)$. 
	If $\rho\colon \Ac\Gc_N^0\to TN$ is the anchor, then we also have $\rho(\alpha(t))=\dot c(t)$ for the following reason. 
	Since ${\bf t}(\gamma(t))=c(t)$ and ${\bf s}(\gamma(t))=y$, 
	one has ${\bf t}_y\circ R_{\gamma(t)}={\bf t}_{c(t)}$ 
	(cf. the equality ${\bf t}_x\circ R_g={\bf t}_y$ in the proof of  Proposition~\ref{invarianttarget_right}\eqref{invarianttarget_right_item2}). 
	Therefore ${\bf t}_y={\bf t}_{c(t)}\circ R_{\gamma(t)^{-1}}\colon \Gc_N^0(y,-)\to M$ 
	with $R_{\gamma(t)^{-1}}(\gamma(t))=\1_{{\bf t}(\gamma(t))}=\1_{\bf c(t)}$. 
	Computing the differential of both sides of the above equality at the point $\gamma(t)\in \Gc_N^0(y,-)$ and using the definition of $\rho$ (cf. Remark~\ref{summary}) we obtain 
	$$T({\bf t}_y)=
	\rho\vert_{(\Ac\Gc_N^0)_{c(t)}}\circ T_{\gamma(t)}(R_{\gamma(t)^{-1}}) \colon T_{\gamma(t)}(\Gc_N^0(y,-))\to T_{c(t)}M.$$
	Evaluating this equality at $\dot\gamma(t)\in  T_{\gamma(t)}(\Gc_N^0(y,-))$ we obtain ${\dot c}(t)=\rho(\alpha(t))$, and this completes the proof of the fact that $\alpha$ is  an $\Ac\Gc_N^0$-lift of~$c$. 
	
%
   
{\bf Step 2}:
 The  $\mathcal{AG}_{{N}}$-orbit of $x$ is contained in its $\mathcal{G}_{{N}}^0$-orbit.
 
To prove this, let $y$ be arbitrary in the $\mathcal{AG}_{{N}}$-orbit of $x$. 
We will prove that there exists $g\in\mathcal{G}_{{N}}^0$ with 
${\bf s}(g)=\1_x$ and ${\bf t}(g)=\1_y$, using an idea from the proof of \cite[Prop. 1.1]{CraFe03}. 
  
To this end, note that from Lemma \ref{smooth-lift} it follows that there exists a smooth curve $\alpha:[0,1]\ap \mathcal{AG}_{{N}}$ whose projection $c:=\pi\circ\alpha $ on ${N}$ is a smooth $\mathcal{AG}_{{N}}$-admissible curve, satisfying $c(0)=x$ and $c(1)=y$,  
and for which there exists an open subset $U\subseteq {N}$ for which $c([0,1])\subseteq U$ and the restriction of $\mathcal{AG}_{{N}}$  to $U$ is trivializable.  
If $E$ is the typical fiber of $\mathcal{AG}_{{N}}$, without loss of generality, we can identify this restriction with $U\times E$.
We can then write  $\alpha(t)=(c(t), u(t))\in U\times E$ for all $t\in[0,1]$ and  we denote by $\{\alpha_t\}_{t\in [0,1]}$ the smooth family of smooth sections of $\mathcal{AG}_{{N}}$ over $U$ defined by
$$(\forall t\in[0,1])(\forall y\in U)\quad \alpha_t(y)=(y, u(t)).$$
Now, for each $t\in [0,1]$, consider the right-invariant vector field 
$\tilde{\alpha}_t$ 
on the restriction of  $\mathcal{G}_{{N}}^0$  to~$U$, 
satisfying $\tilde{\alpha}_t(\1_y)=\alpha_t(y)\in (\mathcal{AG}_{{N}})_y=T_{\1_y}(\Gc_N^0(y,-))$ for every $y\in U$. 
(Recall Remark~\ref{summary}.) 
Let $ \Phi^t_s$ be the flow  of the vector field $\tilde{\alpha}_t$ with initial  conditions $\Phi^t_0({\bf 1}_{c(t)})={\bf 1}_{c(t)}$. 
Since $c([0,1])\subseteq N$ is compact there exist $\delta>0$ and an open subset $W\subseteq\Gc_N$ with $\{\1_z\mid z\in c([0,1])\}\subseteq W\cap N\subseteq\Gc_N^0$ such that $\Phi^t_s\colon W\to \Gc_N^0$ is defined for all 
$s\in[0,\delta]$
and all $t\in[0,1]$.  
For any $t\in[0,1]$, using the fact that $\1_y=\1_{c(a)}\in W$, we define 
$$\gamma (t):=\Phi^t_0(\1_y)\in\Gc_N^0
.$$
Since  $\tilde{\alpha}_t$ is right invariant 
 it follows that $\tilde{\alpha}_t$ is ${\bf s}$-vertical and so
 ${\bf s}(\gamma (t))={\bf s}({\bf 1}_{c(t)})=
 c(t)$ 
 for all $t\in[0,1]$. 
Hence for $g:=\gamma(a)$ one has ${\bf s}(g)=c(0)=x$ and ${\bf t}(g)={\bf t}({\bf 1}_{c(1)})=c(1)=y$, 
as we have wished for.

Now in the most general situation assume that  $x,y\in{N}$  belong  to the same  $\mathcal{AG}_{{N}}$-orbit. 
Then by the Lemma \ref{reparam} there exists a smooth curve $\alpha:[0,1]\ap \mathcal{AG}_{{N}}$  
whose projection $c$ on ${N}$ is an $\mathcal{AG}_{{N}}$-admissible curve with $c(0)=x$ and $c(1)=y$.  

Since  $c([0,1])$ is compact, there exists   a family of open sets  $\{U_i\}_{i=1,\dots,n}$ with the following properties:
\begin{itemize}
\item $\mathcal{AG}_N$ is trivializable over each $U_i$ for ${i=1,\dots,n}$;
 
\item  if $n>1$ then $U_i\cap U_{i+1}\not=\emptyset$ for $i=1,\dots,n-1$ and $U_i\cap U_j=\emptyset$ 
for $1\leq i+1<j\leq n$ and ${i=1,\dots,n-2}$ if $n>2$;
 
\item  $c([0,1])\subset \bigcup\limits_{i=1}^n U_i$.
\end{itemize}
We fix points $t_0=0< t_1\leq\cdots\leq t_i\leq\cdots\leq t_n=1$ in $[0,1]$ with $c(t_i)\in U_i\cap U_{i+1}$ for $i=1,\dots,n-1$ if $n>1$.  
Set $x_i=c(t_i)$ for $i=0,\dots,n$. 
By the arguments above, there exists $g_i\in \mathcal{G}_{{N}}$ with ${\bf s}(g_i)=x_i$ and ${\bf t}(g_i)=x_{i+1}$ for $i=0,\cdots,n-1$. 
Then the product $g=g_{n-1}\cdots g_i\cdots g_0$ is well defined and, by construction, we have ${\bf s}(g)=x$ and ${\bf t}(g)=y$. 
Thus $y$ belongs to the $\mathcal{G}_{{N}}$-orbit of $x$. 


\end{proof}

 


\section{Banach-Lie algebroids and ${\bf s}$-simply connected groupoids}\label{Sect5}

The main result of this section is Theorem~\ref{ssymply}, 
and its proof requires the idea of monodromy group of a foliation. 

\subsection{Monodromy groupoid of a foliation} \label{grdfoliation}

Consider an involutive split subbundle $\mathcal{F}$ of the tangent $TM$ of a connected Banach manifold $M$. If $M$ is modeled on a Banach space $\mathbb{M}$, we have a decomposition $\mathbb{M}=\mathbb{F}\oplus \mathbb{T}$ and from the Frobenius theorem we have a "foliated atlas" $\{(U_\alpha,\phi_\alpha)\}$ on $M$ with the following properties:

 For each $\alpha$, $\phi_\alpha: U_\alpha\ap  \mathbb{M}\equiv \mathbb{F}\times \mathbb{T} $ is a diffeomorphism  with $\phi_\alpha(U_\alpha)=P_\alpha\times Z_\alpha\subset \mathbb{F}\times \mathbb{T}$, and $\phi_\alpha^{-1}(P_\alpha\times \{z\})$ is contained in the  leaf of $\mathcal{F}$ through $\phi_\alpha^{-1}(p,z)$ for any $(p,z)\in P_\alpha\times Z_\alpha$. 
 If we fix some point $x\in U_\a$ then, up to a translation in $\mathbb{T}$, we can always assume that $\phi_\alpha(x)=(p,0)\in \mathbb{F}\times\mathbb{T}$. Then $S_\a=\phi_\a^{-1}(\{p\}\times Z_\a)$ is called a {\it transversal}  at  $x$ and for any $y\in S_\alpha$ if $\phi_\a(y)=(p,z)$ then $\phi_\a^{-1}(P_\a \times\{z\})$ is called  the {\it plaque} through $y$ and of course is contained in the leaf through $y$.
 
 If $U_\alpha\cap U_\beta\not=\emptyset$ then the transition diffeomorphism $\phi_\alpha\circ\phi_\beta^{-1}(p,z)=(f(p,z), h(z))$ is a local diffeomorphism of $\mathbb{F}\times \mathbb{T}$ which respect this product structure.
 
Let $\Pi(\mathcal{F})$ be the set of homotopy classes with fixed end points of piecewise smooth paths contained in leaves of $\mathcal F$. One can define groupoid structure on $\Pi(\mathcal{F})$ in the same way as in the construction of fundamental groupoid $\Pi(M)$, see \ref{funfgrd}.
The differentiable structure  is  obtained by a direct adaptation to the Banach framework, step by step,  of such a structure as it is built   in \cite[Sect. 2]{Ph}. 
We will describe this construction in detail  in our Banach context:

{\bf Step 1}: 
{ \it Consider a leaf $L$ of $\mathcal{F}$ and a  piecewise smooth path $\g\colon [0,1]\to L$, 
	and denote $x_0:=\gamma(0)$ and $x_1:=\gamma(1)$. 
\begin{enumerate}
\item\label{Step1_item1}   
   There exists two charts $(U_i,\phi_i)$ around  $x_i$, $i=0,1$, with $\phi_i(U_i)={P}_i\times Z $ where $P_i$ (resp. $Z$)  is a simply connected open  neighborhood of $0$ in $\mathbb{F}$ (resp. $\mathbb{T}$) and  
 $\phi_i(x_i)=(p_i,0)\in \mathbb{F}\times \mathbb{T}$.   
  
\item\label{Step1_item2}    
   For $i=0,1$ if $S_i:=\phi_i^{-1}(\{p_i\}\times Z)$ then 
   there exists  a continuous map $H_\g: [0,1]\times S_0\ap M$ which is smooth in the second variable and $H_\g(\cdot,y)$ is a piecewise smooth path contained in the leaf through $y\in S_0$ with the following properties:
   $H_\g(\cdot,x_0
   )=\g$;
   $H_\g(0,y)=y$ and $H_\g(1,y)\in S_1$ 
   for all $y\in S_0$; 
   the mapping $S_0\to S_1$, 
  $y\mapsto H_\g(1,y)$, is a diffeomorphism.
\end{enumerate}
} 

\begin{proof} 
Let $\widetilde{S}\subseteq M$ be an embedded submanifold 
	with $x_0\in \widetilde{S}$ and $T_x\widetilde{S}\oplus\Fc_x=T_xM$ for every $x\in\widetilde{S}$. 
 There exist a finite set of  charts $(V_i,\psi_i)$ for $i=0,\dots, n$  of the foliated atlas    and a partition $t_0=0<t_1< \cdots<t_n<1=t_{n+1}$ satisfying the following conditions: 
\begin{itemize}
\item 
$\psi_i(V_i)=P_i\times Z_i$ where $P_i$ (resp. $Z_i$) is a simply connected open neighborhood of  $0\in\mathbb{F}$ (resp. $0\in\mathbb{T}$) for $i=0,\dots, n$;
\item
 $\g(t_i)\in V_{i-1}\cap V_{i}$ for $i=1,\dots, n$, $\g(0)=x_0\in V_0$, 
 and $\g(1)=x_1\in V_n$;  
\item
 $\g([t_{i},t_{i+1}])\subseteq P_i\times\{0\}\subseteq V_i$ for $i=0,\dots,n$.
\end{itemize}
By suitable translations in $\TT$ and $\FF$ we may assume that there exist $p_i\in P_i$ with $\psi(x_i)=(p_i,0)\in P_i\times Z_i$ for $i=0,1$, 
and we may also select the foliation chart $(V_0,\psi_0)$ such that $S_0:=\psi_0^{-1}(\{p_0\}\times Z_0)$  
is an open neighborhood of $x_0\in\widetilde{S}$ 

We will now prove by induction on $n$ that if one has a family of local charts satisfying the above conditions then there exists a continuous mapping~$H_\gamma$ with the properties required in \eqref{Step1_item2} above, with $(U_0,\phi_0):=(V_0,\psi_0)$ and $(U_1,\phi_1):=(V_{n+1},\psi_{n+1})$. 

\emph{Case $n=0$:} 
We denote $\psi:=\psi_0$, $V:=V_0$, $P:=P_0$, and $Z:=Z_0$ for simplicity. 
Using the path  $\widehat{\g}:=\psi\circ \g\colon [0,1]\to P\equiv P\times\{0\}$,  
define 
$$\widehat{H}\colon [0,1]\times Z\ap P\times Z, \quad \widehat{H}(t,z):=(\widehat{\g}(t),z).
$$
Obviously $\widehat{H}$ is continuous and  piecewise smooth (resp. smooth) in $t$ (resp. $z$) and the path $t\mapsto \widehat{H}(t,z)$ is contained in $P\times\{z\}$. 

The charts $(U_0,\phi_0)=(U_1,\phi_1):=(V,\psi)$ satisfy  condition~\eqref{Step1_item1}.  
For  $j=0,1$ and 
$S_j:=\psi^{-1}(\{p_j\}\times Z)$, 
one has a bijection $\psi_{p_j}\colon S_j\to Z$ satisfying $\psi(y)=(p_j,\psi_{p_j}(y))$ for all $y\in S_j$, 
 and then we can define 
$$H_\g: [0,1]\times S_0\to M,\quad 
H_\g(t,y):=\psi^{-1}(\widehat{H}(t,\psi_{p_0}(y)))
=\psi^{-1}(\widehat{\gamma}(t),\psi_{p_0}(y)).$$
Hence $H_\g(0,y)
=\psi^{-1}(p_0,\psi_{p_0}(y))
=y$ for all $y\in S_0$.  
For every $y\in S_0$ one then has 
$$\psi(H_\gamma(1,y))
=(\widehat{\gamma}(1),\psi_{p_0}(y)) 
=(\psi(x_1),\psi_{p_0}(y))
=(p_1,\psi_{p_0}(y))$$
hence 
$$H_\gamma(1,y)=\psi^{-1}(p_1,\psi_{p_0}(y))=\psi_{p_1}^{-1}(\psi_{p_0}(y)). $$
Since  $\psi_{p_j}\colon S_j\to Z$ is a global chart of $S_j$ for $j=1,2$,  
the above equality shows that the mapping 
$S_0\to S_1$, $y\mapsto H_\gamma(1,y)$ is a diffeomorphism that maps $p_0\in S_0$ to $p_1\in S_1$. 
Condition~\eqref{Step1_item2}  from Step~1 above is thus satisfied as well. 

\emph{Induction step:}
Assume that $n\ge 1$ and the assertion has been proved for $n-1$. 

Then, on the one hand,  by the induction hypothesis applied for the path $\gamma':=\gamma\vert_{[0,t_n]}$, 
one obtains a continuous map $H_{\gamma'}: [0,t_n]\times S_0\ap M$ which is smooth in the second variable and $H_{\gamma'}(\cdot,y)$ is a piecewise smooth path contained in the leaf through $y\in S_0$ with the following properties:
$H_{\gamma'}(\cdot,x_0
)=\gamma'$;
$H_{\gamma'}(0,y)=y$ and $H_{\gamma'}(t_n,y)\in S_{n-1}:=\psi_{n-1}^{-1}(\{p_{n-1}\}\times Z)$ 
for all $y\in S_0$; 
the mapping 
$$\Psi\colon S_0\to S_{n-1},\quad y\mapsto H_{\gamma'}(t_n,y)$$ 
is a diffeomorphism. 

On the other hand, by the Case $n=0$ applied for the path $\gamma'':=\gamma\vert_{[t_n,1]}$ and with the above embedded submanifold $S_{n-1}\subseteq M$ in the role of $\widetilde{S}$, 
one obtains a continuous map $H_{\gamma''}: [t_n,1]\times S_{n-1}\ap M$ which is smooth in the second variable and $H_{\gamma''}(\cdot,y)$ is a piecewise smooth path contained in the leaf through $y\in S_{n-1}$ with the following properties:
$H_{\gamma''}(\cdot,\gamma(t_{n-1}))=\gamma''$;
$H_{\gamma''}(t_n,y)=y$ and $H_{\gamma''}(1,y)\in S_n:=\psi_n^{-1}(\{p_n\}\times Z)$ 
for all $y\in S_{n-1}$; 
the mapping $S_{n-1}\to S_n$, 
$y\mapsto H_{\gamma''}(1,y)$, is a diffeomorphism. 

We now define $H_\gamma\colon[0,1]\times S_0\to M$ by 
$$H_\gamma(t,y)
=\begin{cases}
H_{\gamma'}(t,y) & \text{ if }t\in[0,t_n],\\
H_{\gamma''}(t,\Psi(y)) & \text{ if }t\in[t_n,1]
\end{cases}
$$
for all $y\in S_0$. 
This definition is correct since 
$H_{\gamma'}(t_n,y)=\Psi(y)=H_{\gamma''}(t_n,\Psi(y))$. 
It is straightforward to check the other properties of $H_\gamma$ required in the statement of Step~1 with the charts $(U_0,\phi_0):=(V_0,\psi_0)$ and $(U_n,\phi_n):=(V_n,\psi_n)$ , and this completes the proof. 
\end{proof}

{\bf Step 2}: {\it 
For any piecewise smooth path $\g\colon [0,1]\to M$ contained in a leaf, use Step~1 to define
$$\begin{aligned}
\mathcal{V}(\g,  (U_0,\phi_0),(U_1,\phi_1), & H_\gamma) \\
                :=\{[\g']\in  \Pi(\mathcal{F})\mid & {\bf s}(\g')\in S_0,\  {\bf t}(\g')\in S_1,\  
								[\g']
								 =[\a_{\g'}\star H_\g(\cdot,y)\star \b_{\g'}]\\
								& \text{for some }y\in S_0\} 
\end{aligned}$$
where $\a_{\g'}$ is a path which joins ${\bf s}(\g')$ to $H_\g(0,y)=y$ and lies in the plaque $\phi_0^{-1}(P_0\times \{z_0(y)\})$ through~$y\in S_0$, 
while and $ \b_{\g'} $ is a path which joins $H_\g(1,y)$ to ${\bf t}(\g')$ and lies in the plaque $\phi_1^{-1}(P_1\times \{z_1(y)\})$ 
through $H_\g(1,y)\in S_1$, 
where  $\phi_j(y)=(p_j,z_j(y))\in P_j\times Z$ for $j=0,1$. 
Then there exists a unique topology on $\Pi(\mathcal{F})$ 
for which 
the set of all the above sets $\mathcal{V}(\g, (U_0,\phi_0),(U_1,\phi_1),H_\g)$ 
is a neighborhood base for $[\gamma]\in\Pi(\mathcal{F})$}. 

\begin{proof}
For every $[\gamma]\in\Pi(\mathcal{F})$ let us denote by $\Bc([\gamma])$ the set of all the above subsets $\mathcal{V}(\g, (U_0,\phi_0),(U_1,\phi_1),H_\g)$ of $\Pi(\mathcal{F})$. 
Then the assertion will follows as soon as we will have checked that the following properties: 
\begin{enumerate}[{\rm (BP1)}]
	\item\label{BP1} For every $[\gamma]\in\Pi(\mathcal{F})$ one has $\Bc([\gamma])\ne\emptyset$, and $[\gamma]\in\Vc$ for all $\Vc\in\Bc([\gamma])$. 
	\item\label{BP2} If $[\gamma]\in\Pi(\mathcal{F})$, $\Vc\in\Bc([\gamma])$ and $[\gamma']\in\Vc$, then there exists $\Vc'\in\Bc([\gamma'])$ with $\Vc'\subseteq\Vc$.  
	\item\label{BP3} If $[\gamma]\in\Pi(\mathcal{F})$ and $\Vc,\Vc'\in\Bc([\gamma])$, then there exists $\Vc''\in\Bc([\gamma])$ with $\Vc''\subseteq\Vc\cap\Vc'$. 
\end{enumerate}
(See for instance \cite[Prop. 1.2.3]{En89}.) 

The property~(BP\ref{BP1}) directly follows by Step~1. 


For~(BP\ref{BP2}), let  $[\g' ] \in \mathcal{V}(\g, (U_0,\phi_0),(U_1,\phi_1),H_\g)$. 
We can assume that   
${\g'}$ is the path $\a_{\g'}\star H_\g(\cdot,y)\star \b_{\g'}$ for some $y\in S_0$. 
Since $\phi_i(U_i)= P_i\times Z$, if $\phi_0(y)=(0,z)$ we have $\phi_0({\g}'(0))=({p}_0,z)$ and $\phi_1({\g}'(1))=({p}_1,H_\g(1,y))$. 
By 
composition of $\phi_i$ by  translation in $\mathbb{F}$ and $\mathbb{T}$  and after shrinking  the open set $U_i$ if necessary,  we get a new foliated chart $(U'_i,\phi'_i)$ such that  $U'_i\subset U_i$,  $\phi'_i\circ {\g}'(i)=({p}'_i,0) $ and $\phi_i(U'_i)=P'_i\times Z'_i$ with $Z'_i\subset  Z_i$ for $i=0,1$. 
If  $S'_i=(\phi'_i)^{-1}(\{p'_i\}\times Z'_i)$ as in the proof of 
Step~1 we can build  a map $H_{{\g}'}:[0,1]\times S'_0\ap M$ with properties of Step~1\eqref{Step1_item2} with respect to $(U'_i,\phi'_i)$ for $i=0,1$.
We have 
$$[\g']\in  \mathcal{V}({\g}', (U'_0,\phi'_0),(U'_1,\phi'_1),H_{{\g}'})\subseteq  \mathcal{V}(\g, (U_0,\phi_0),(U_1,\phi_1),H_\g).$$

For~(BP\ref{BP3}), let 
$$\Vc:=\mathcal{V}(\g, (U_0,\phi_0),(U_1,\phi_1),H_\g)
\text{ and }
\Vc'=\mathcal{V}(\g, (U'_0,\phi'_0),(U'_1,\phi'_1),H'_{\g}).$$ 
By a method similar to the one used for~(BP\ref{BP3}) above, 
one can then find a local chart $(U''_j,\phi''_j)$ at $\gamma(j)\in M$ 
with $U''_j\subseteq U_j\cap U'_j$ for $j=0,1$ and 
with 
$$\Vc(\g, (U''_0,\phi''_0),(U''_1,\phi''_1),H_{\g}'')\subseteq \Vc\cap\Vc'$$
and this completes the proof of Step~2. 
\end{proof}

{\bf Step 3}: {\it The  set of all the basic open sets $\mathcal{V}(\g, (U_0,\phi_0),(U_1,\phi_1),H_\g)$ defines a smooth atlas on $\Pi(\mathcal{F}) $  modeled on $\mathbb{F}\times\mathbb{T}\times \mathbb{F}$. 
The source map ${\bf s}$ and the target ${\bf t}$ are 
submersions, and the inversion  ${\bf i}$ and the multiplication ${\bf m}$ are also smooth maps.}

\begin{proof}
We already noted that $\mathcal{V}(\g, (U_0,\phi_0),(U_1,\phi_1),H_\g)$ depends only of the homotopy class $[\g]\in \Pi(\mathcal{F}) $. 
Consider the map 
$$\begin{aligned}
& \Phi_\g: \mathcal{V}(\g, (U_0,\phi_0),(U_1,\phi_1),H_\g)\ap \mathbb{F}\times\mathbb{T}\times \mathbb{F},\\
& \Phi_\g([\mu])=(\phi_0(\mu(0)), \phi_1(H_\g(1,\mu(0))),\phi_1(\mu(1))).
\end{aligned}$$
It is clear that $\Phi$ is well defined and injective  and depends only of the homotopy class of $\g$. 
In addition, since $\phi_i(U_i)=P_i\times Z$ is a product of simply connected open sets, it follows from the proof of Step~1  that $\Phi_\gamma$ is injective. 

For $i=0,1$ if  $P'_i\subseteq P_i$ and $Z'\subseteq Z$ are open sets  then $\Phi^{-1}_\g(P'_0\times Z'\times P'_1)$ is the basic set  $\mathcal{V}(\g, (U'_0,\phi'_0),(U'_1,\phi'_1),H'_\g)$ where $U'_i=\phi_i^{-1}(P'_i\times Z')$,  $\phi'_i={\Phi_i}\vert_{ U'_i}$ for $i=0,1$ and $H'_\g$ is the restriction of $H_\g$ to $[0,1]\times S'_0$ with $S'_0=S_0\cap U'_0$ and so $\Phi_\g$ is continuous. 
It follows by the  arguments of the previous proof that $\Phi_\g$ is open and so $\Phi_\g$ is a homeomorphism. It is obvious that  each transition  map $\Phi_\g\circ \Phi_{\g'}^{-1}$ between charts is smooth.

Now as the smoothness is a local property, it is obvious that ${\bf s}$, ${\bf t}$ and ${\bf i}$ restricted to a chart $(\mathcal{V}(\g, (U_0,\phi_0),(U_1,\phi_1),H_\g),\Phi_\g)$ are smooth. 
For the multiplication ${\bf m}$ it is also easy to verify the smoothness in local charts in $\Pi(\mathcal{F})\times \Pi(\mathcal{F}) $ and $\Pi(\mathcal{F})$. 
The splitness properties of the tangent maps of  ${\bf s}$ and ${\bf t}$ is also clear from the definition of $\Phi_\g$ on a chart. 
\end{proof}

\subsection{Integrable Banach-Lie algebroids and ${\bf s}$-simply connected groupoids}\label{tildeG} 
As in finite dimension, we say that a Banach-Lie algebroid $\mathcal{A}$ is {\it integrable} if there exists a n.n.H Banach-Lie groupoid $\mathcal{G}$ such that $\mathcal{A}$ is isomorphic to the algebroid~$\mathcal{AG}$. 
If this is the case, then we say that the Banach-Lie algebroid $\mathcal{A}$ \emph{integrates to} the n.n.H. Banach-Lie groupoid $\mathcal{G}$. 

The purpose of this subsection is to prove the following theorem which  generalizes a well known result in finite dimensions 
and shows that every integrable Banach-Lie algebroid integrates to certain ${\bf s}$-simply connected n.n.H. Banach-Lie groupoid. 
(See \cite[Prop. 6.6]{MoMr}.)

\begin{theorem} \label{ssymply} 
	If 
	$\mathcal{G}\tto M$ is  a n.n.H. Banach-Lie groupoid, 
	then there exist an ${\bf s}$-simply connected n.n.H. Banach-Lie groupoid $\tilde{\mathcal{G}}\tto M$ 
	and a morphism of Banach-Lie groupoids $\Phi: \tilde{\mathcal{G}}\ap \mathcal{G}$ over $\id_M$, 
	which is a local diffeomorphism
	and for which $\Ac\Phi\colon\mathcal{A}\tilde{\mathcal{G}}\to\mathcal{AG}$ 
	is an isomorphism of Banach-Lie algebroids over~$\id_M$. 
\end{theorem}


\begin{proof} This is an adaptation to our context of the proof of 
\cite[Th. 1.31]{CraFe}. 
At first note  that we can prove the result  for the sub-groupoid associated to each connected component of $M$.  
Therefore from now, assume that $\mathcal{G}\tto M$ is a n.n.H. Banach-Lie groupoid, where $M$ is a connected (hence pure) Banach manifold. 
From the construction of the algebroid $\mathcal{AG}$ and using the local normal form of a submersion in the Banach framework 
(that is, submersion charts), with the arguments of the proof of 
\cite[Prop. 1.30]{CraFe} without loss of generality, we can also assume that $\mathcal{G}$ is ${\bf s}$-connected.

For any $x\in M$ denote by $\tilde{\mathcal{G}}(x,-)$ the universal covering of $\mathcal{G}(x,-)$ and we set
$$\tilde{\mathcal{G}}=\dis \bigcup_{x\in M}\tilde{\mathcal{G}}(x,-).$$
In fact $\tilde{\mathcal{G}}$ is the set of homotopy classes of paths in each $\mathcal{G}(x,-)$ (with fixed end points) and starting at ${\bf 1}_x$. 

Consider the foliation $\mathcal{F}_{\bf s}$ of $\mathcal{G}$ defined by the fibration ${\bf s}:\mathcal{G}\ap M$. 
The construction of the n.n.H. Banach-Lie Banach groupoid  structure on the monodromy groupoid of a (regular)  foliation ({\it cf}. subsection  \ref{grdfoliation}) used only local arguments. 
Therefore the set $\Pi(\mathcal{F}_{\bf s})$ of homotopy classes with fixed end of path contained in leaves
of $\mathcal{F}_{\bf s}$ has  a structure of n.n.H. Banach manifold and moreover, since the smoothness is a local property the source ${\bf s}([\g])=\g(0)$ and the target ${\bf t}([\g])=\g(1)$ are  
submersions from $\Pi(\mathcal{F}_{\bf s})$ onto $\mathcal{G}$. 

We also define ${\tilde{\bf m}}([\g_1],[\g_2])$ as the homotopy class of the concatenation of $\g_2$ with $R_{\g_1(1)}\circ \g_1$.
Just as in the finite dimensional context we thus obtain a n.n.H. Banach-Lie groupoid structure on $\tilde{\mathcal{G}}$. 

Now,  since  ${\bf s}:\mathcal{G}\ap M$ is a submersion and $M$ is connected, from the normal form of a submersion in Banach framework (i.e., submersion charts), it follows that for any $g\in \mathcal{G}$ there exists a neighborhood $\mathcal{U}$  which is diffeomorphic to a product of open sets $U\times V\subset \mathbb{M}\times \mathbb{K}$  where the Banach space $\mathbb{M}$ (resp. $\mathbb{K}$) is the model space of $M$ (resp. of the typical fiber  of ${\bf s}$).  
In the same way, since $\tilde{\bf s}:\tilde{\mathcal{G}}\ap M$ is a  submersion, any  $[\g]\in \tilde{\mathcal{G}}$ has a neighborhood $\tilde{\mathcal{U}}$ diffeomorphic to a product of open $\tilde{U}\times\tilde{V}\subset \mathbb{M}\times \tilde{\mathbb{K}}$ where  $\tilde{\mathbb{K}}$ is the typical Banach model of the typical fiber  of $\tilde{\bf s}$. 
But the typical fiber of $\tilde{\bf s}$ is the universal covering of the typical fiber  of ${\bf s}$ so  $\tilde{\mathbb{K}}$ is isomorphic to $\mathbb{K}$.
 But, from the construction of the smooth structure on $\Pi(\mathcal{F}_{\bf s})$ and the characterization of $\tilde{\mathcal{G}}$ as submanifold of $\Pi(\mathcal{F}_{\bf s})$, for each $[\g]\in  \tilde{\mathcal{G}}$ one has an open neighborhood of type $\mathcal{V}(\g, (U_0,\phi_0),(U_1,\phi_1),H_\g)\cap \tilde{\mathcal{G}}$ and the restriction of $\Phi_\g$ to this open set is $\Phi_\g([\mu])=(\phi_1(H_\g(1,\mu(0))),\mu(1))$ which is a diffeomorphism on an open set of $\tilde{U}\times\tilde{V}\subset \mathbb{M}\times \tilde{\mathbb{K}}$. 
 It follows that  the map $\Phi: \tilde{\mathcal{G}}\ap \mathcal{G}$ given by $\Phi([\g])=\g(1)$ is a local diffeomorphism. 
 From the construction of the Banach-Lie algebroid  this last property implies that  
 $\Ac\Phi\colon\mathcal{A}\tilde{\mathcal{G}}\to\mathcal{AG}$ is an isomorphism of Banach-Lie algebroids, and this completes the proof. 
\end{proof}

\section{Locally transitive Banach-Lie groupoids and transitive algebroids}\label{Sect6}
 \subsection{Locally transitive  Banach-Lie  groupoids} 
 As usual for topological group\-oids, we will say that a Banach-Lie groupoid $\mathcal{G}\tto M$ is \emph{locally transitive} 
if each orbit $\Gc.x$ is open in $M$. 
This condition is equivalent to the following one: 
\begin{equation}\label{loctran_alt}
\text{For every $x\in M$ the map ${\bf t}_x\colon \Gc(x,-)\to M$ is a submersion.}
\end{equation}
In fact, we know that  ${\bf t}_x:\Gc(x,-)\ap \Gc.x$ is a submersion since it is a principal bundle 
by Theorem \ref{5.4}\eqref{smooth4_item2}. If $\Gc.x$ is open in $M$ so is a submanifold of $M$ whose tangent space at any point coincides to the tangent space of $M$ at that point and  so  ${\bf t}_x:\Gc(x,-)\ap M$ is a submersion.

\begin{lemma}\label{loctransitive}
If $\Gc\tto M$ is a  locally transitive Banach-Lie groupoid, 
then the following assertions hold: 
\begin{enumerate}[(i)]
\item\label{loctransitive_item1} 
Every orbit of $\Gc$ is a union of some connected components of~$M$. 
\item If an orbit of $\Gc$ is connected as a topological subspace of $M$ 
(for instance if it is the orbit of a point $x\in M$ whose ${\bf s}$-fiber $\Gc(x,-)$ is connected), 
then that orbit coincides with some connected component of~$M$. 
\item If some connected component of $M$ is invariant under the action of $\Gc$ on $M$, 
then that connected component coincides with an orbit of~$\Gc$.  
\end{enumerate}
\end{lemma}

\begin{proof}
It suffices to prove Assertion~\eqref{loctransitive_item1} 
and then the other assertions follow directly. 

Since for arbitrary $x\in M$ its orbit $\Gc.x:={\bf t}(\Gc(x,-))$ is open  and 
 $M$ is a disjoint union of the orbits of~$\Gc$, it then follows that every orbit is also closed 
because its complement in $M$ is the union of the other orbits, which is open. 
Thus every orbit of $\Gc$ is simultaneously open and closed, and then 
it is a union of some connected components of~$M$.
\end{proof}

Extending the finite-dimensional case, we have the following result. 

\begin{theorem}\label{transitive} 
Let $\mathcal{G}\tto M$ be a Banach-Lie groupoid. 
Then the following properties are equivalent:
\begin{enumerate}[(i)]
\item\label{transitive_item1} The groupoid $\mathcal{G}\tto M$ is locally transitive.
\item\label{transitive_item2}  The map $({\bf s},{\bf t}):\mathcal{G}\ap M\times M$ 
is  a submersion. 
\item\label{transitive_item3} The orbits of $\mathcal{G}$ are unions of connected components of $M$.
\item\label{transitive_item4} The anchor $\rho:\mathcal{AG}\ap TM$ is split and surjective.
\end{enumerate}
If the above conditions are satisfied, then $\Gc$ is a split Banach-Lie groupoid and the map  ${\bf t}_x: \mathcal{G}(x,-)\ap \Gc.x$ 
is a $\mathcal{G}(x)$-principal bundle for every $x\in M$. 
\end{theorem}


\begin{proof}

\eqref{transitive_item1}$\Rightarrow$\eqref{transitive_item2}:	
Denoting 
$$\alpha:=({\bf s},{\bf t}):\mathcal{G}\ap M\times M$$ 
we must check that the map 
$T_g\alpha\colon T_g\Gc\to T_xM \times T_yM$ is surjective and its kernel 
is a split subspace of 
$T_g\Gc$ for arbitrary $g\in\Gc$ with $\alpha(g)=:(x,y)$. 

To this end we first note that the hypothesis implies via~\eqref{loctran_alt} that $\Gc(u,v)={\bf t}_u^{-1}(v)$ is a submanifold of $\Gc$ for all $u,v\in M$ with $\Gc(u,v)\ne\emptyset$. 
That is, the n.n.H. Banach-Lie groupoid $\Gc\tto M$ is split. 
Moreover, 
$$\ker T_g\alpha=T_g(\alpha^{-1}(\alpha(g)))=T_g(\alpha^{-1}(x,y))
=T_g(\Gc(x,y))$$
hence this is a 
split 
subspace of $T_g\Gc$ 
since we have just seen that $\Gc(x,y)$ is a submanifold of~$\Gc$.

It remains to prove that the map 
$T_g\alpha\colon T_g\Gc\to T_xM \times T_yM$ is surjective. 
To this end it suffices to check that the range of $T_g\alpha$ contains both linear subspaces  $T_xM \times \{0\}$ and $\{0\} \times T_yM$. 
Our present hypothesis~\eqref{transitive_item1} ensures via \eqref{loctran_alt} that the  map ${\bf t}\vert_{\Gc(x,-)}\colon\Gc(x,-)\to M$ is a submersion, 
hence $T_g{\bf t}(T_g(\Gc(x,-)))=T_yM$.  
This further implies $T_g{\bf t}(T_g\Gc)=T_yM$, 
hence $T_g\alpha(T_g\Gc)\supseteq \{0\}\times T_yM$. 

On the other hand, since ${\bf t}\vert_{\Gc(y,-)}\colon\Gc(y,-)\to M$ is a submersion and ${\bf s}={\bf t}\circ {\bf i}$, where ${\bf i}\colon\Gc\to\Gc$ is a diffeomorphism, we obtain that 
${\bf s}\vert_{\Gc(-,y)}\colon\Gc(-,y)\to M$ is a submersion. 
Then, as above, we obtain $T_g\alpha(T_g\Gc)\supseteq T_xM\times\{0\}$, and we are done. 

\eqref{transitive_item2}$\Rightarrow$\eqref{transitive_item1}: 
Let $g\in\Gc$ be arbitrary and denote $\alpha(g)=:(x,y)\in M\times M$. Since $\alpha$ is a submersion, there exist open sets $U,V\subseteq M$ with $x\in U$ and $y\in V$, for which there exists a smooth map 
$\sigma\colon U\times V\to\Gc$ with $\sigma(x,y)=g$ and $\alpha\circ\sigma=\id_{U\times V}$. 
In particular, for every $v\in V$ we obtain ${\bf s}(\alpha(x,v))=x$ and ${\bf t}(\alpha(x,v))=v$. 
Therefore we obtain the well-defined smooth map 
$$\tau\colon V\to \Gc(x,-), \quad \tau(\cdot):=\alpha(x,\cdot),$$ 
which is defined on the neighborhood $V$ of $y\in M$ and satisfies 
${\bf t}\circ \tau=\id_V$. 
Since $y\in\Gc.x$ is arbitrary, we thus see that the groupoid $\Gc\tto M$ is locally transitive. 



\eqref{transitive_item1}$\Rightarrow$\eqref{transitive_item3}: 
This implication is exactly 
Lemma~\ref{loctransitive}\eqref{loctransitive_item1}.

 
\eqref{transitive_item3}$\Rightarrow$\eqref{transitive_item1}: 
This is clear. 

\eqref{transitive_item1}$\Leftrightarrow$\eqref{transitive_item4}:  
For every $x\in M$ one has 
\begin{equation}\label{transitive_proof_eq1}
\rho\vert_{(\Ac\Gc)_x}=T_{\1_x}({\bf t}\vert_{\Gc(x,-)})
\colon T_{\1_x}(\Gc(x,-))\to T_xM
\end{equation} 
(see for instance Remark~\ref{summary}), 
and then it is clear that \eqref{transitive_item4} is equivalent to \eqref{loctran_alt}, which is further equivalent 
to~\eqref{transitive_item1}. 

 Now the last part is a direct application of Theorem \ref{5.4}\eqref{smooth4_item2} and the fact that each $\Gc_N\tto N$ is a principal bundle for any connected component $N$ of $M$. 
\end{proof}


\begin{remark}\label{15Aug2016}
\normalfont
We note for later use that if $\mathcal{G}\tto M$ is a topological groupoid 
that satisfies the condition~(BLG1) of Definition~\ref{BLG} 
(i.e., $ \mathcal{G}$ is a n.n.H. Banach manifold and $M$ is a Banach manifold)
and for which the map $({\bf s},{\bf t}):\mathcal{G}\ap M\times M$
is a submersion, then the condition~(BLG2) follows automatically, 
that is, the map ${\bf s}:\mathcal{G}\ap M$ is a submersion. 
\end{remark}

\begin{remark}\label{19Aug2017}
\normalfont 
Let $\Gc\tto M$ be a Banach-Lie groupoid 
and denote by $M_0$ the set of all $x\in M$ 
whose corresponding map
${\bf t}_x: \mathcal{G}(x,-)\ap M$ is a submersion. 
It is clear that $M_0$ is an open subset of $M$. 
It follows by the discussion after~\eqref{loctran_alt} 
that $M_0$ is the set of all points in $M$ whose $\Gc$-orbits are open. 
In particular, $M_0$ is an open $\Gc$-invariant subset of~$M$, 
and the groupoid $\Gc\tto M$ is locally transitive if and only if $M_0=M$. 

The condition $M_0=M$ cannot be weakened to the condition that $M_0$ intersects every connected component of $M$. 
That is, if 
for every connected component ${N}$ of $M$ there exists $x\in {N}$ whose corresponding map
${\bf t}_x: \mathcal{G}(x,-)\ap {N}$ is a submersion, 
then the groupoid $\Gc\tto M$ need not be locally transitive. 
Examples in this connection are provided by any action of a Lie group on a connected manifold $A\colon G\times M\to M$ with a dense open nontrivial orbit. 
For instance, the tautological action of the group of all invertible matrices $G=\GL(n,\RR)$ on $M=\RR^n$ has its orbits $\RR^n\setminus\{0\}$ and $\{0\}$, 
and the corresponding groupoid $M\times G\tto M$ 
(see subsection~\ref{action}) is not locally transitive. 
A wider perspective on this problem is offered by the following discussion on transitive groupoids. 
\end{remark}

\subsection{Transitive Banach-Lie groupoids}
A Banach-Lie groupoid $\mathcal{G}\tto M$ is called  \emph{transitive} if the map 
$({\bf s},{\bf t}):\mathcal{G}\ap M\times M$
is a surjective submersion.

It follows by Theorem~\ref{transitive} that 
if $\mathcal{G}\tto M$ is a Banach-Lie groupoid  
then the following properties are equivalent:
\begin{enumerate}[(i)]
\item\label{1transitive_item1} 
The groupoid $\mathcal{G}\tto M$ is transitive.
\item\label{1transitive_item2}  $\Gc$ is split and the maps $T({\bf s},{\bf t}):T\mathcal{G}\ap TM\times TM$ and 
$({\bf s},{\bf t}):\mathcal{G}\ap M\times M$ are surjective.
\item\label{1transitive_item3} The groupoid $\mathcal{G}$ is locally transitive and has only one orbit, namely $M$.
\end{enumerate}
If moreover the base $M$ is connected, then it has only one connected component, 
hence the above conditions are equivalent to the 
fact that the groupoid is locally transitive, 
which is further equivalent to any of the following properties: 
\begin{enumerate}[(i)]
\setcounter{enumi}{3}
\item\label{1transitive_item4} The anchor $\rho:\mathcal{AG}\ap TM$ is split and surjective.
\item\label{1transitive_item5}  
For every $x\in M$ the map 
${\bf t}_x: \mathcal{G}(x,-)\ap M$ is a  surjective submersion.
\end{enumerate}
If this is the case, then for every $x\in M$ the map  ${\bf t}_x: \mathcal{G}(x,-)\ap M$ is a $\mathcal{G}(x)$-principal bundle. 

Consider a locally transitive Banach-Lie groupoid $\mathcal{G}\tto M$. Then  the restriction $\mathcal{G}_N\tto N$ to any connected component $N$ of $M$ is a  Banach-Lie groupoid whose algebroid $\mathcal{AG}_N$ is the restriction of $\mathcal{AG}$ to $N$. Therefore  $\mathcal{G}_N\tto N$ is a transitive Banach-Lie groupoid which is the gauge groupoid of the principal bundle ${\bf t}_x: \mathcal{G}_N(x,-)\ap N$  for any $x\in N$.

\subsection{Banach-Lie algebra bundles and transitive Banach-Lie algebroids}
Motivated by \cite[Def. 3.3.8]{Ma} 
we define a {\it  Banach-Lie algebra bundle}  
(for short an \emph{LAB}) as a Lie algebroid $(L,M,\rho,[\cdot,\cdot])$ with anchor $\rho\equiv 0$, 
such that for each $x\in M$ there exists a local trivialization $\Psi:\pi^{-1}(U)\ap U\times{\bf \mathfrak{g}}$ with $x\in U$, where ${\bf \mathfrak{g}}$ is a Banach-Lie algebra, and the restriction $\Psi_y:=\Psi\vert_{\pi^{-1}(y)}\colon\pi^{-1}(y)\ap \{y\}\times {\bf \mathfrak{g}}$ 
is a Lie algebra isomorphism for all $y\in U$. 

A {\it morphism } of LAB is a bundle morphism which is a morphism of Lie algebras between each fiber.

As in \cite[Prop. 3.3.9]{Ma}, any characteristic subalgebra ${\bf \mathfrak{h}}$ of ${\bf \mathfrak{g}}$  
(i.e., $\varphi({\bf \mathfrak{h}})={\bf \mathfrak{h}}$ for all automorphism $\varphi$ of ${\bf \mathfrak{g}}$) generates a sub-LAB $(K,M,[\cdot,\cdot])$
of $(L,M,[\cdot,\cdot])$. 
In particular if ${\bf \mathfrak{h}}$ is the center $Z{\bf \mathfrak{g}}$ (resp. the derived ideal $[{\bf \mathfrak{g}},{\bf \mathfrak{g}}]$) of ${\bf \mathfrak{g}}$ we get an associated sub-LAB denoted $ZL$ (resp.$[L,L]$).

For any Banach bundle $E\ap M$, we denote by $\End(E)$ the bundle over $M$ of bundle morphisms of $E$. 
By same argument as in \cite[Sect. 3.3]{Ma}, $\End(E)$ is a LAB  with typical fiber $\End({\bf \mathfrak{g}})$ and provided with the classical  bracket $[\Phi,\Phi']=\Phi\circ\Phi'-\Phi'\circ\Phi$. 
The   subalgebra $\Der({\bf \mathfrak{g}})$ of derivations of ${\bf \mathfrak{g}}$ of $\End({\bf \mathfrak{g}})$ is a characteristic algebra and so we get sub-LAB denoted $\Der(L)$ of $L$. 
The map $\ad: L\ap \Der(L)$, 
where $\ad: L_x\ap \Der(L_x)$ is the usual adjoint action in Lie algebras, is a morphism of LAB 
(same arguments as in the proof of \cite[Prop. 3.3.10]{Ma}). 
Therefore $\ad(L)$ is a sub-LAB of $\Der(L)$ 
called the {\it adjoint} LAB of $L$. 

Motivated by 
Theorem~\ref{transitive}\eqref{transitive_item4}, 
a split Banach-Lie algebroid $(\mathcal{A},M,\rho,[\cdot,\cdot])$ is called \emph{transitive} if its anchor $\rho$ is surjective. 
In this case, as in finite dimension we have the following result. 

\begin{proposition}
If $(\mathcal{A},M,\rho,[\cdot,\cdot])$ is a transitive   Banach-Lie algebroid, then $\ker\rho$ is a Banach subbundle of $ \mathcal{A}$. 
\end{proposition}

\begin{proof}
Denote $\mathcal{K}=\ker\rho$ and fix some  $x\in M$. 
By hypothesis we have $\mathcal{A}_x=\ker \rho_x\oplus F_x$ for a suitable closed subspace $F_x\subseteq\Ac_x$, and $\rho_x$ is surjective. 
It follows that the restriction ${\rho_x}\vert_{ F_x}$ is an isomorphism onto $T_xM$. 
Choose a trivialization $\Psi\colon \mathcal{A}_U\ap \phi(U)\times \mathbb{A}$ such that $(U,\phi)$ is a local chart at~$x$. 
We can identify $\mathcal{A}_x$ with the typical fiber $\mathbb{A}$ of $\Ac$ on the connected component of $M$ that contains~$x$. 
We simply denote by $\mathbb{K}:=\ker \rho_x$ and  $\mathbb{F}:=F_x$. 
Therefore we can also identify $\mathbb{A}$ with $\mathbb{K}\times \mathbb{F}$. Moreover for $x$ fixed,  $\mathbb{F}$ is isomorphic to the model space $\mathbb{M}$ of $M$ at $x$, 
and then the  trivialization can be viewed as a map  $\Psi\colon\mathcal{A}_U\ap U\times \mathbb{K}\times \mathbb{M}$. On the other hand $T\phi:TM\vert_{ U}\ap \phi(U)\times \mathbb{M}$ is a trivialization of $TM$ over $U$. 
Since we have to build a local trivialization of $\mathcal{A}$ whose restriction induces a local trivialization for~$\mathcal{K}$,  without loss of generality, we  may assume that $U$ is an open subset of $\mathbb{M}$,  $TM\vert_{ U}$ is the trivial bundle $U\times \mathbb{M}$ and  $\mathcal{A}_{U}$ is the trivial bundle $U\times \mathbb{K}\times \mathbb{M}$. 
With this notation, $\rho$ can be written as a map $(y, u)\mapsto (y,\rho_y(u))$ where $y\mapsto \rho_y$  is a map from $U$ to $L(\mathbb{K}\times \mathbb{M},\mathbb{M})$ and moreover the restriction of $\rho_x$ to $\mathbb{M}$ belongs to $GL(\mathbb{M})$. 
Thus after shrinking $U$ if necessary, we may assume that the restriction of $\rho_y$ to $\mathbb{M}$  belongs to $GL(\mathbb{M})$ for any $y\in U$. 
We set
$\sigma_y=-({\rho_y}\vert_{ \mathbb{M}})^{-1}\circ \rho_y$. Then the map $y\mapsto \s_y$ is a smooth map from $U$ to $L(\mathbb{K}\times \mathbb{M},\mathbb{M})$
and has  the following properties:

$\bullet\;\;$  ${\sigma_y}\vert_{\mathbb{M}}=-Id_\mathbb{M}$

 $\bullet\;\;$  $\ker\sigma_y=\ker \rho_y=({\rho_y}\vert_{ \mathbb{M}})^{-1}\circ \rho_y(\mathbb{K})$.

 It follows that $\kappa_y:(k, m)\mapsto (k,\sigma_y(k,m))$  is an isomorphism from $\mathbb{K}\times \mathbb{M}$ to $(\ker \rho_y)\times \mathbb{M}$ and $y\mapsto \kappa_y$ is a smooth map from $U$ to $GL(\mathbb{A})$. This implies that $(y,u)\mapsto (y, \kappa_y^{-1}(u))$ is a trivialization of $U\times \mathbb{A}$ whose restriction to $\mathcal{K}\vert_{ U}\equiv \dis\cup _{y\in U}\ker \rho_y$ defines a  trivialization of   $\mathcal{K}\vert_{ U}$.
 \end{proof}

Now given a transitive   Banach-Lie algebroid $(\mathcal{A},M,\rho,[\cdot,\cdot])$ since the restriction of $\rho$ to $\mathcal{K}$ is null and the bracket in restriction to  
global (resp. local) sections of $\mathcal{K}$ takes values in the module of global (resp. local) sections of $\mathcal{K}$ this implies that  $\mathcal{K}$ is a $LAB$ over $M$.

In the case of the algebroid $\mathcal{AG}$ of  Banach-Lie groupoid $\mathcal{G}$ over $M$, according to section \ref{transitive}, the algebroid $\mathcal{AG}$ is transitive if and only if $\mathcal{G}$ is locally transitive. Moreover, the restriction $\mathcal{G}_{N}$ to any connected component ${N}$ of $M$ has a structure of principal bundle whose structural Banach-Lie group is the typical model $G_{N}$ of the isotropy group (in $\mathcal{G}_{N}$) of any point $x\in {N}$ and so   $G_{N}$ is a gauge groupoid of this principal bundle.

\subsection{Atiyah  exact sequence  of a principal bundle}
A particularly important transitive Banach-Lie groupoid is the gauge groupoid of a Banach principal bundle. 
Therefore we will  look for a  famous short exact sequence of Banach algebroids canonically associate to this context: 
the \emph{Atiyah exact sequence}. 

Let $\pi:P\ap M$ be a Banach principal bundle with structural  Banach-Lie group $G$ and $\tau:E\ap P$ be a Banach bundle over $P$. Assume that there exists a smooth right action $E\times G\ap E$, $(\xi,g)\mapsto \xi g$ such that
\begin{enumerate} 
\item\label{H1}  $\xi\mapsto \xi g$ is an bundle isomorphism over the right translation $R_g:P\ap P$.
\item\label{H2} $E$ is covered by equivariant trivializations in the sense that around each $u_0\in P$ there is an open set of the form $\mathcal{U}= \pi^{-1}(U)$ where $U$ is a neighborhood of $\pi(u_0)$ in $M$ and a Banach bundle chart
$$\psi:\mathcal{U}\times \mathbb{E} \ap E_{\mathcal{U}}$$
which is equivariant in the sense that $\forall u \in \mathcal{U},\;\xi\in \mathbb{E}, \;  g \in G$ then 
$$\psi(ug, \xi)= \psi(u, \xi)g.$$
\end{enumerate}
Under these assumptions we have the following result. 

\begin{proposition}\label{EG} 
The quotient set $E/G$ has  a canonical structure of Banach bundle $\hat{\tau}: E/G\ap M$ such that the natural projection  $q:E\ap E/G$ is a surjective submersion and a bundle morphism over $\pi: P\ap M$. 
Moreover $\tau:E\ap P$ can be identified with the pull back of $ \hat{\tau}: E/G\ap M$  by $\pi:P\ap M$. 
\end{proposition}

The method of proof is the same as that of \cite[Prop. 3.1.1]{Ma05} and we only give the key points which are essential  in this Banach context. 

\begin{proof}
We denote by $\hat{\xi}$ the $G$-orbit of any $\xi\in E$, and then we define 
$$\hat{\tau}: E/G\ap M,\quad 
\hat{\tau}(\hat{\xi})=(\pi\circ \tau)(\xi).$$ 
This map is well defined 
since for all $\xi\in E$ and $g\in G$ one has 
$\pi(\tau(\xi g))=\pi(\tau(\xi)g)=\pi(\tau(\xi))$, 
where the first equality follows from the above hypothesis~\eqref{H1} 
while the second equality follows by the fact that $\pi\colon P\to M$ is a aprincipal bundle with its structural group~$G$. 

We now define a Banach structure on each fiber $\hat{\tau}^{-1}(x)$. 
If $\bar{\xi}$ and $\bar{\eta}$ belongs to  $\bar{\tau}^{-1}(x)$ then there exists $g\in G$ such that ${\xi g}$ and ${\eta}$  such that $\tau(\xi g)=\tau(\eta)$. 
Thus the sum $\xi g+\eta$ is well defined and we can define 
$\hat{\xi}+\hat{\eta}=\widehat{\xi g+\eta}$  and $\lambda\hat{\xi}=\widehat{\lambda\xi}$. 
It is easy to see that these operations are well defined and endow $\hat{\tau}^{-1}(x)$ with a vector bundle structure.
Now  from assumption~\eqref{H2} preceding the statement of Proposition~\ref{EG}, the action of $G$ on $E$ must be proper. Indeed, consider a net $\{(g_j,\xi_j)\}_{j\in J}$ in $G\times E$ for which there exists $\lim\limits_{j\in J}(g_j.\xi_j,\xi_j)=:(\eta,\xi)\in E\times E$. We set  $u_j=\tau(\xi_j)$, $v_j=\tau(g_j.\xi_j)$, $u=\tau(\xi)$ and $v=\tau(\eta)$. 
Therefore we have $\tau(\lim\limits_{j\in J}(g_j.\xi_j,\xi_j))=\lim\limits_{j\in J}(\tau(g_j.\xi_j),\tau(\xi_j))=(v,u)$. 
From the properness of the action of $G$ on $P$, it follows that $\lim\limits_{j\in J}g_j=:g $ exists in $G$  and so from Lemma \ref{smooth1}, the action of $G$ on $E$ is proper. 
Then Lemma~\ref{popertop} implies  that the quotient space~$E/G$ is a Hausdorff space  and the natural projection $q:E\rightarrow E/G$ is continuous.  
For any $u\in\pi^{-1}(x)$ the restriction $q_u: E_u\rightarrow (E/G)_x$ is a linear map which is surjective. 
The  assumption~\eqref{H2} 
implies that $q_u$ is injective and then $q_u$ is an (algebraic) isomorphism. 
But $E_u$ is topological subspace of $E$ so $q_u: E_u\rightarrow E/G$ is continuous. 
Therefore $(E/G)_x$ can be provided with  a Banach space structure isomorphic to  $E_u$. Moreover since $q_u$ is an isomorphism the inclusion of $(E/G)_x$  in $E/G$ is continuous. 

Next we must show that $\hat{\tau}: E/G\rightarrow M$ is a Banach bundle. Fix some $x_0\in M$  and $u_0\in E_{x_0}$. According to assumption (2), consider an equivariant trivialization $\psi:\mathcal{U}\times \mathbb{E} \ap E_{\mathcal{U}}$  so that  around  $u_0$ which can be chosen so that $\mathcal{U}$ is a trivialization of $P$ which is isomorphic to $U\times G$. For simplicity we can identify $\mathcal{U}$ with $U\times G$. Then consider the map
$$\psi^G: U\times \mathbb{E}\longrightarrow (E/G)\vert_{ U}$$
 given by $\psi^G(x,\xi)=\widehat{\psi(x,e,\xi)}$. Clearly $\psi^G$ is an injective continuous map. The    restriction  of $\psi^G$ to $\{x\}\times \mathbb{E}$ is nothing else but $q_u$ according to the identification  $\mathcal{U}\equiv U\times G$ and so is an isomorphism.  From the assumption (2) and the fact that the action of $G$ on $P$ is proper,  it follows easily that $\psi^G$ is a homeomorphism 
 and so we get a chart on $E/G$. Now  consider two such charts $\psi_i^G: U_i\times \mathbb{E}\longrightarrow (E/G)\vert_{ U_i}  $ for $i=1,2$ with  $(E/G)\vert_{ U_1}\cap  (E/G)\vert_{ U_2}\not=\emptyset$ then we have 
  $$(\psi^G_1)^{-1}\circ \psi^G_2(x,\xi)=(\psi_1)^{-1}\circ\psi_2(x,e,\xi)$$
  Therefore the set of charts of type $\psi^G: U\times \mathbb{E}\to (E/G)\vert_{ U}$ defines a Banach manifold structure on $E/G$ which is also a Banach bundle structure on $M$.
  Finally by construction we have $\hat{\tau}\circ q=\pi\circ \tau$ and since $q_u$ is an isomorphism it follows that $E\to P$ is the pullback of $\hat{\tau}:E/G\ap M$ over $\pi$. 
\end{proof} 

\subsection*{Application of Proposition \ref{EG} for $E=TP$}
As in finite dimensions 
(see 
\cite[\S 3.2]{Ma05}), there is a natural right  action of $G$ on $TP$ by $(X,g)\in T_uP\times G\mapsto TR_g(X)$, 
where $R_g$ is the right translation by $g$ on $P$. 
This action satisfies  the assumptions \eqref{H1}--\eqref{H2} above. 
Indeed \eqref{H1} 
is obvious and for \eqref{H2} 
we choose any open set $U$ such that $P_U$ can be identified with $U\times G$. 
Then, if  $\mathfrak{g}$ is the Lie algebra of $G$, the tangent bundle  $TP$ can be identified with $(U\times \mathbb{M})\times( G\times \mathfrak{g})$  over $U\times G$. 
Then $X\in T_{(x,\g)}P$ can be written as $(x,\g,\dot{x},\mathfrak{X})$, and $R_g(x,\g)=(x,\g.g)$. 
Hence 
\begin{eqnarray}\label{TRg}
TR_g(X)\equiv (x,\g.g,\dot{x},\Ad_{g^{-1}}\mathfrak{X})
\end{eqnarray}
which clearly implies assumption~\eqref{H2}. 


Therefore applying Proposition \ref{EG} to $TP$ we get the following diagram
\[\begin{tikzcd}
TP \arrow[r, "q"] \arrow[d,"\tau"]
&TP/G \arrow[d,"\hat{\tau}"] \\
 P  \arrow[r, "\pi"] 
 &M
\end{tikzcd}
\]
The vector bundle $TP/G$ is known as the  {\bf Atiyah bundle} and was firstly introduced by Atiyah in \cite{At}.
Since $TP\ap P$ is the pullback of $TP/G\ap M$ each local or global section of this last bundle 
gives rise to a  local or global pull-back section of  $TP\ap P$ which is $G$-invariant and conversely.
Therefore the $G$ invariant local or global vector fields on $P$ can be identified with sections of $TP/G\ap M$. 

Now recall that the gauge groupoid of $P$  is the quotient set $(P\times P)/G$ (see Example \ref{gauge}). 
Since any  invariant local vector fields on $P$ can be identified with local   sections of $TP/G\ap M$, 
from the construction of the algebroid of a groupoid, 
it follows that the Lie algebroid associated to the groupoid $(P\times P)/G$ is exactly $TP/G\ap M$ 
and the associated anchor will denoted $\rho:TP/G\ap TM$ in the sequel. 
In fact $\rho$ is induced from the canonical map $T\pi:TP\ap TM$ in restriction to $G$-invariant vectors 
and so $\rho$ is surjective.

 On the one hand  consider the vertical  subbundle $VP\ap M$ of $TP\ap M$ that   is the kernel of $T\pi$. 
 Since $\pi\circ R_g=\pi$ for any $ g\in G$ it follows that $VP$ is $G$-invariant.

 On the other hand  
 as in finite dimension by formally same arguments we can show that  $VP$  is isomorphic to  $P\times \mathfrak{g}$. 
 (See for instance the proof of \cite[Prop. 3.2.2]{Ma05}.)
  Now according to \eqref{TRg}  we can identify $VP$ with $P\times\mathfrak{g}$ with the action $(x,\g,\mathfrak{X})g=(x,\g.g,\Ad_{g^{-1}}\mathfrak{X})$ and so the assumption of Proposition \ref{EG} are satisfied for $E=P\times \mathfrak{g}$ we get a  Banach bundle $(P\times \mathfrak{g})/G\ap M$ which is in fact a subbundle of $TP/G\ap M$ which is the kernel of $\rho$. Note that   $(P\times \mathfrak{g})$  provided with    the induced bracket $[\cdot,\cdot]$ of the Lie algebroid $TP/G$ is  a LAB. 
 We finally the following exact sequence of Banach bundles over $M$
$$\begin{tikzcd}
       0\arrow{r} &(P\times \mathfrak{g})/G\arrow{r}{j} & TP/G\arrow{r}{\rho} & TM\arrow{r} & 0
\end{tikzcd}$$
which is called the {\bf Atiyah sequence} of $P$.

\section{Perspectives on generalized inverses in Banach algebras
}\label{Sect7}

In this final section of the present paper we show that the preceding theory of Banach-Lie groupoids sheds fresh light on  the generalized inverses in Banach algebras, in particular on Moore-Penrose inverses in $C^*$-algebras, a research area that has been rather active. 
From the extensive literature that is available, the most relevant references for our present paper include 
\cite{HM92}, \cite{Ko01}, \cite{AC04}, \cite{ACM05}, \cite{Boa06}, \cite{ArCG08}, \cite{LeRo12}, \cite{ArM13}. 

We begin by a general construction of groupoids associated to semigroups, which we will afterwards specialize to the multiplicative semigroups underlying the associative Banach algebras.

\subsection{Groupoids associated to semigroups}

\begin{lemma}\label{S1}
Let $(A,\cdot)$ be an arbitrary semigroup 
and define 
$$
Q(A):=\{a\in A\mid a^2=a\} 
\text{ and }
\Gc(A):=\{(a,b)\in A\times A\mid aba=a,\ bab=b\}.
$$ 
Then one has a groupoid $\Gc(A)\tto Q(A)$ with its source/target maps 
$${\bf s},{\bf t}\colon \Gc(A)\to Q(A),\quad 
{\bf s}(a,b):=ba,\ {\bf t}(a,b):=ab$$ 
with its multiplication $(a_1,b_1)\cdot(a_2,b_2):=(a_1a_2,b_2b_1)$ if ${\bf s}(a_1,b_1)={\bf t}(a_2,b_2)$, 
and with its inversion map $(a,b)\mapsto(b,a)$. 
\end{lemma}

\begin{proof}
We only need to check that the source/target maps and the multiplication indeed take values in $Q(A)$ and $\Gc(A)$, respectively. 
For the source and target maps we note that if $(a,b)\in A\times A$ satisfy $aba=a$ and $bab=b$, then $(ab)^2=ab$ and $(ba)^2=ba$, 
hence $ab,ba\in Q(A)$. 

For the multiplication, the condition ${\bf s}(a_1,b_1)={\bf t}(a_2,b_2)$ is equivalent to $b_1a_1=a_2b_2$, 
and then one obtains $(a_1a_2,b_2b_1)\in\Gc(A)$ since 
$$(a_1a_2)(b_2b_1)(a_1a_2)=a_1a_2b_2(a_2b_2)a_2=a_1(a_2b_2)^2 a_2=a_1(a_2b_2)a_2=a_1a_2$$
where the first equality follows by $b_1a_1=a_2b_2$, the third equality follows by 
the property $(a_2b_2)^2=a_2b_2$ (which is a consequence of $(a_2,b_2)\in\Gc(A)$ as we have already seen above), 
and the fourth equality follows by $(a_2,b_2)\in\Gc(A)$. 
\end{proof}

\begin{lemma}\label{S2}
Let $A$ be a $*$-semigroup 
and define 
$$
P(A):=\{a\in Q(A)\mid a^*=a\} \text{ and }
\Jc(A):=\{a\in A \mid aa^*a=a\}.$$ 
Then one has a groupoid $\Jc(A)\tto P(A)$ with its source/target maps 
$${\bf s},{\bf t}\colon \Jc(A)\to P(A),\quad 
{\bf s}(a):=a^*a,\ {\bf t}(a):=aa^*$$ 
with its multiplication obtained as the restriction of the multiplication of~$A$ 
and with its inversion map $a\mapsto a^*$. 
Moreover there is the injective morphism of groupoids $\Jc(A)\to\Gc(A)$, $a\mapsto (a,a^*)$. 
\end{lemma}

\begin{proof}
The structure $\Jc(A)\tto P(A)$ is a groupoid by \cite[\S 4.2, Th. 3]{Lw98}, 
and it is straightforward that the map $a\mapsto (a,a^*)$ is an injective morphism of groupoids. 
\end{proof}

\subsection{Generalized inverses and groupoids associated to Banach algebras}

For any associative algebra $A$, an element $a\in A$ is called \emph{regular} if $a\in aAa$, and if this is the case 
then every element $b\in A$ with $a=aba$ is called a \emph{generalized inverse} of $a$. 
The generalized inverse of a regular element is not uniquely determined in general, and it is therefore difficult to extend the classical continuity and differentiability properties of the inversion mapping from invertible elements to regular elements. 

A way out of the above difficulty is to regard $A$ as a multiplicative semigroup, and to consider the set $\Gc(A)$ from Lemma~\ref{S1}, 
that is, the set of all pairs of regular elements $(a,b)\in A\times A$ for which $b$ is a generalized inverse of~$a$ and $a$ is a generalized inverse of~$b$. 
When $A$ is a Banach algebra, the differential geometry of the set $\Gc(A)$ was investigated in \cite{ACM05}, 
and we will show below that the corresponding results have their natural place in the theory of Banach-Lie groupoids. 
To this end we specialize the construction of Lemma~\ref{S1} for the multiplicative semigroups 
defined by associative Banach algebras. 
See for instance \cite[App. A]{B06} and the references therein for real analytic mappings on Banach manifolds. 

\begin{theorem}\label{S3}
If $A$ is a unital associative Banach algebra, then $\Gc(A)\tto Q(A)$ is a real analytic Banach-Lie groupoid. 
Moreover, this groupoid is locally transitive and its isotropy group at $\1\in Q(A)$ is the Banach-Lie group $A^\times$ of invertible elements of~$A$. 
\end{theorem}

\begin{proof}
It follows by \cite[Cor. 1.3 and Th. 1.5]{ACM05} and \cite[Cor. 1.5]{CPR90} that both 
$\Gc(A)$ and $Q(A)$ are real analytic submanifolds of $A$, 
and in particular their tangent spaces at any point are 
split closed linear subspaces of $A$. 
Since the multiplication map $A\times A\to A$, $(a,b)\mapsto ab$, 
is clearly real analytic, 
it then follows by \cite[5.8.5]{Bo09} that the structure maps of the groupoid $\Gc(A)\tto Q(A)$ 
are real analytic. 

Moreover, the map $({\bf s},{\bf t}):\Gc(A)\to  P(A)\times P(A)$ is a submersion by 
\cite[Prop. 1.12]{ACM05} and \cite[Th. 2.1]{CPR90}, 
hence by Remark~\ref{15Aug2016} we obtain that $\Gc(A)\tto P(A)$ is a locally transitive Banach-Lie groupoid. 
\end{proof}

\begin{remark}
	\label{S3.5}
\normalfont
From the perspective of Theorem~\ref{S3}, it is natural to ask if, in the case when $A$ is endowed with a continuous involution, 
the corresponding groupoid $\Jc(A)\tto P(A)$ given by Lemma~\ref{S2} 
is a Banach-Lie groupoid and moreover if it is a Banach-Lie subgroupoid of $\Gc(A)\tto Q(A)$. 
For general associative Banach $*$-algebras, it is not difficult to check that $P(A)$ is a submanifold of $Q(A)$,  
since it is the fixed-point set of the involutive diffeomorphism $a\mapsto a^*$ of $Q(A)$. 
However it is less clear how $\Jc(A)$ should be given a manifold structure with respect to which the source/target maps of the groupoid $\Jc(A)\tto P(A)$ would be submersions. 

We will see below that the above questions can be satisfactorily answered in the important case of $C^*$-algebras, 
but it would be interesting to understand what happens for other important examples of  associative Banach $*$-algebras, 
as for instance the restricted Banach algebra related to the restricted Grassmann manifold from \cite[Sect. 6.2]{PS86}. 
(See also \cite[Sect. 6]{BTR07} for the contrast between the restricted Banach algebra and the $C^*$-algebras.)
\end{remark}

\subsection{The special case of $C^*$-algebras}

The following definition was suggested by \cite{OS11} and is a specialization of the construction from Lemma~\ref{S2}. 
(We recall that in the special case of the matrix algebra $\Ag=M_n(\CC)$ or bounded operators on Hilbert space $\Ag=B(\Hc)$ there is a one-to-one correspondence from the orthogonal projections $p=p^2=p^*\in\Ag$ onto the linear subspaces of~$\CC^n$ or $\Hc$, and this why the set of orthogonal projections defined below for any $C^*$-algebra~$\Ag$ is called the Grassmann manifold associated with~$\Ag$, see also e.g. \cite{PoRe87}.)
We slightly change the notation of Lemma~\ref{S2} in order to emphasize the importance of the class of $C^*$-algebras 
among the Banach $*$-algebras. 

\begin{definition}\label{def1}
\normalfont
For any $C^*$-algebra $\Ag$ 
we introduce the following subsets: 
\begin{itemize}
\item $\Pc(\Ag):=\{p\in\Ag\mid p=p^2=p^*\}$ (the \emph{Grassmann manifold} of $\Ag$); 
\item $\Uc(\Ag):=\{a\in\Ag\mid aa^*,a^*a\in\Pc(\Ag)\}$ (the set of \emph{partial isometries} in $\Ag$). 
\end{itemize}
The \emph{groupoid associated to $\Ag$} is $\Uc(\Ag)\tto\Pc(\Ag)$ with the following structure maps: 
\begin{itemize}
\item the target/source maps ${\bf t},{\bf s}\colon \Uc(\Ag)\to\Pc(\Ag)$, ${\bf t}(a)=aa^*$, ${\bf s}(a)=a^*a$; 
\item the inversion map ${\bf i}\colon\Uc(\Ag)\to\Uc(\Ag)$, ${\bf i}(a)=a^*$; 
\item the composition defined on $\Uc(\Ag)^{(2)}:=\{(a,b)\in\Uc(\Ag)\times\Uc(\Ag)\mid {\bf s}(a)={\bf t}(b)\}$ by $\mu\colon\Uc(\Ag)^{(2)}\to\Uc(\Ag)$, $\mu(a,b):=ab$. 
\end{itemize}

If $\Bg$ is another $C^*$-algebra and $\varphi\colon\Ag\to\Bg$ is a $*$-morphism, 
then $\Uc(\varphi):=\varphi\vert_{\Uc(\Ag)}$. 
\end{definition}

\begin{theorem}\label{S4}
For any unital $C^*$-algebra $\Ag$ its corresponding groupoid $\Uc(\Ag)\tto\Pc(\Ag)$ is a real analytic Banach-Lie groupoid. 
Moreover, this groupoid is locally transitive. 
The isotropy group of the above groupoid at $\1\in\Pc(\Ag)$ is the Banach-Lie group of unitary elements of~$\Ag$. 
\end{theorem}

\begin{proof}
It follows by \cite[Prop. 3.3]{ACM05} and \cite[(3)]{PoRe87} that 
both $\Uc(\Ag)$ and $\Pc(\Ag)$ are real analytic submanifolds of $\Ag$, 
and in particular their tangent spaces at any point are 
split closed linear subspaces of $\Ag$. 
Since the multiplication map $\Ag\times\Ag\to\Ag$, $(a,b)\mapsto ab$, 
and the adjoint $\Ag\to\Ag$, $a\mapsto a^*$, 
are clearly real analytic, 
it then follows by \cite[5.8.5]{Bo09} that the structure maps of the groupoid $\Uc(\Ag)\tto\Pc(\Ag)$ 
are real analytic. 

Moreover, the map $({\bf s},{\bf t}):\Uc(\Ag)\to  \Pc(\Ag)\times \Pc(\Ag)$ is a submersion by \cite[Prop. 3.4]{ACM05}, 
hence by Remark~\ref{15Aug2016} we obtain that $\Uc(\Ag)\tto\Pc(\Ag)$ is a locally transitive Banach-Lie groupoid. 

Finally, the isotropy group of the groupoid $\Uc(\Ag)\tto\Pc(\Ag)$ at $\1\in\Pc(\Ag)$ is 
$$(\Uc(\Ag))(\1)=\{a\in\Uc(\Ag)\mid {\bf s}(a)={\bf t}(a)=\1\}
=\{a\in\Ag\mid a^*a=aa^*=\1\}$$
which is exactly the unitary group of~$\Ag$. 
Since $\Uc(\Ag)\tto\Pc(\Ag)$ is a Banach-Lie groupoid, it follows by Theorem~\ref{5.4}\eqref{smooth4_item1} that all its isotropy groups are Banach-Lie groups. 
However, in the special case of the unitary group of a $C^*$-algebra $\Ag$, it is well known that this is a Banach-Lie group. 
This follows for instance from the fact that the unitary group of $\Ag$ is an algebraic subgroup (of degree $\le 2$) of the group of invertible elements of~$\Ag$, hence one can use \cite[Th. 4.13 and Ex. 2.21]{B06}.
This completes the proof. 
\end{proof}

\begin{remark}
\normalfont
Let $\mathbb{CSTAR}$ be the category of $C^*$-algebras and $\mathbb{GRPD}$ be the category of Banach-Lie groupoids. 
Then it is easily checked that the correspondence 
$\Uc\colon\mathbb{CSTAR}\to\mathbb{GRPD}$ is a functor. 

Since the above functor takes values in the category of Banach-Lie groupoids, 
we can also compose it with the functor that associates to every Banach-Lie groupoid its 
Lie algebroid. 
In this connection we note that the differentiable structures of the source-fibers of the groupoid $\Uc(\Ag)$ 
were discussed in~\cite{AC04}. 
\end{remark}

\subsection{Moore-Penrose inverse in $C^*$-algebras}
The research on Moore-Penrose inverses  in $C^*$-algebras and even in more general Banach algebras 
has been rather active. 
We will briefly discuss here the relation between some of these recent results and the theory of Banach-Lie groupoids that we developed in this paper. 
In particular, we show that a part of the operator theoretic research in this area can be cast in a natural way in the framework of groupoids. 

For the sake of simplicity we will discuss here Moore-Penrose invertibility only in $C^*$-algebras. 
If $\Ag$ is a unital $C^*$-algebra, its set of regular elements is denoted by 
$$\Ag^\dagger:=\{a\in\Ag\mid a\in a\Ag a\}.$$
It follows by \cite[Th. 6]{HM92} that $\Ag^\dagger$ is exactly the set of all $a\in\Ag$ for which there exists a \emph{Moore-Penrose inverse}, that is, a unique element $a^\dagger\in\Ag$ satisfying 
$$aa^\dagger a=a,\ a^\dagger aa^\dagger=a^\dagger,\ 
(a^\dagger a)^*=a^\dagger a,\ (aa^\dagger)^*=aa^\dagger.$$
It then follows that for every $a\in\Ag^\dagger$ one has $a^\dagger\in\Ag^\dagger$ and $(a^\dagger)^\dagger=a$. 
We recall from \cite[Ex. 1.1]{Ko01} that in general $\Ag^\dagger$ is \emph{not} an open subset of~$\Ag$. 
It is also known that although the mapping 
$\Ag^\dagger\to\Ag^\dagger$, $a\mapsto a^\dagger$, is well-defined, in general it is not continuous (cf. \cite[Sect. 3]{LeRo12}). 
Also, if $a,b\in\Ag^\dagger$ then we may have $ab\not\in\Ag^\dagger$. 
These observations show that both the algebraic and analytic structures of the set~$\Ag^\dagger$ are pathological in some sense, unlike the group $\Ag^\times$ of all invertible elements, which is always an open subset of $\Ag$ and is a Banach-Lie group. 

Despite the above aspects, it is clearly desirable  to have a framework in which  the Moore-Penrose inversion has better continuity and differentiability properties. 
One possible approach to that problem is to understand the  
relation between the Moore-Penrose inversion and the locally transitive Banach-Lie groupoids 
from Theorems \ref{S3}~and~\ref{S4}.  
To conclude this paper, we take a very first step in that direction. 

\begin{proposition}\label{final}
For every unital $C^*$-algebra $\Ag$ the following assertions hold: 

\begin{enumerate}[{\rm(i)}]
	\item\label{final_item1}
The mappings 
$\eta\colon \Ag^\dagger\to \Gc(\Ag)$, $a\mapsto (a,a^\dagger)$, 
and 
$\pi\colon \Gc(\Ag)\to \Ag^\dagger$, $(a,b)\mapsto a$, 
are well defined and $\pi\circ \eta=\id_{\Ag^\dagger}$. 
\item\label{final_item2} 
For any sequence $\{a_n\}_{n\ge 1}$ in $\Ag^\dagger\setminus\{0\}$ and any $a\in \Ag^\dagger\setminus\{0\}$ 
with $\lim\limits_{n\to\infty}a_n=a$ 
one has 
$\lim\limits_{n\to\infty}\eta(a_n)=\eta(a)$ if and only if 
$\lim\limits_{n\to\infty}{\bf s}(\eta(a_n))={\bf s}(\eta(a))$.  
\item\label{final_item3} 
One has $\Uc(\Ag)\subseteq \Ag^\dagger$ and the map $\eta\vert_{\Uc(\Ag)}\colon \Uc(\Ag)\to \Gc(\Ag)$ is an injective morphism of Banach-Lie groupoids. 
\end{enumerate}
\end{proposition}

\begin{proof}
Assertion~\eqref{final_item1} is straightforward. 

Assertion~\eqref{final_item2} follows by \cite[Th. 1.6]{Ko01}. 
	
Assertion~\eqref{final_item3} follows by the well-known fact that every partial isometry in $\Ag$ is a regular element and, more exactly, one has $\Uc(\Ag)=\{a\in\Ag^\dagger\mid a^\dagger=a^*\}$. 
This completes the proof. 
\end{proof}

\subsection*{Acknowledgment}

We wish to thank Anatol Odzijewicz and Aneta Sli\.zewska for their helpful suggestions at the beginning of this project. 
We extend our thanks to the Referee for numerous suggestions that greatly helped us improve the exposition and even some results of our paper. 
The authors D.~Belti\c t\u a and F.~Pelletier also acknowledge financial support from the Centre Francophone en Math\'ematiques de Bucarest and the GDRI ECO-Math.



\begin{thebibliography}{99999999}

\bibitem[An11]{An11}
M.~Anastasiei, 
{\it Banach-Lie algebroids}. 
An. \c Stiin\c t. Univ. Al. I. Cuza Ia\c si. Mat. (N.S.) {\bf 57} (2011), no. 2, 409--416.

\bibitem[At57]{At}
M.F.~Atiyah,  
{\it Complex analytic connections in fibre bundles}. 
Trans. Amer. Math. Soc. 85 (1957), 181--207.


\bibitem[AC04]{AC04}
E.~Andruchow, G.~Corach, 
{\it Differential geometry of partial isometries and partial unitaries}. 
Illinois J. Math. 48 (2004), no. 1, 97--120. 

\bibitem[ACM05]{ACM05}
E.~Andruchow, G.~Corach, M.~Mbekhta, 
{\it On the geometry of generalized inverses}. 
Math. Nachr. 278 (2005), no. 7--8, 756--770. 

\bibitem[ACG08]{ArCG08}
M.L.~Arias, G.~Corach, M.C.~Gonzalez, 
{\it Generalized inverses and Douglas equations}. 
Proc. Amer. Math. Soc. {\bf 136} (2008), no. 9, 3177--3183. 

\bibitem[AM13]{ArM13}
M.L.~Arias, M.~Mbekhta, 
{\it $A$-partial isometries and generalized inverses}. 
Linear Algebra Appl.  {\bf 439} (2013), no. 5, 1286--1293. 


\bibitem[BG08]{BaGa}
M.~Baillif,  A.~Gabard
{\it Manifolds: Hausdorffness versus homogeneity}. 
Proc. Amer. Math. Soc. {\bf 136} (2008), no. 3, 1105--1111.

\bibitem[Be06]{B06}
D.~Belti\c t\u a, 
``Smooth homogeneous structures in operator theory''. 
Chapman \& Hall/CRC Monographs and Surveys in Pure and Applied Mathematics, 137. Chapman \& Hall/CRC, Boca Raton, FL, 2006. 

\bibitem[BTR07]{BTR07}
D.~Belti\c t\u a, T.S.~Ratiu, A.B.~Tumpach, 
{\it The restricted Grassmannian, Banach Lie-Poisson spaces, and coadjoint orbits}. 
J. Funct. Anal. {\bf 247}  (2007), no. 1, 138--168. 

\bibitem[Boa06]{Boa06}
E.~Boasso, 
{\it On the Moore-Penrose inverse in $C^*$-algebras}. 
Extracta Math.  {\bf 21} (2006), no. 2, 93--106. 






\bibitem[Bou71a]{Bo71}
N.~Bourbaki, 
``Topologie g\'en\'erale''. Chapitres 1 \`a 4. 
Hermann, Paris, 1971. 

\bibitem[Bou71b]{Bo09}
N.~Bourbaki,
``Vari\'et\'es diff\'erentielles et analytiques''. 
Fascicule de r\'esultats. 
Springer, 1971. 

\bibitem[Bou72]{Bo72}
N.~Bourbaki, 
``Groupes et alg\`ebres de Lie''. Chapitres II et III. 
Actualit\'es Scientifiques et Industrielles, No. 1349. Hermann, Paris, 1972.



\bibitem[CP12]{CaPe}
P.~Cabau, F.~Pelletier, 
{\it Almost Lie structures on an anchored Banach bundle}. 
J. Geom. Phys. {\bf 62} (2012), no. 11, 2147--2169.

\bibitem[CPR90]{CPR90}
G.~Corach, H.~Porta, L.~Recht, 
{\it Differential geometry of systems of projections in Banach algebras}. 
Pacific J. Math.  {\bf 143} (1990), no. 2, 209--228.

\bibitem[CF03]{CraFe03}
M.~Crainic, R.L.~Fernandes,
{\it  Integrability of Lie brackets}. 
Ann. of Math. (2)  {\bf 157} (2003), no. 2, 575--620. 

\bibitem[CF11]{CraFe}
M.~Crainic, R.L.~Fernandes,
{\it Lectures on integrability of Lie brackets}.  
In: ``Lectures on Poisson geometry'', Geom. Topol. Monogr., 17, Geom. Topol. Publ., Coventry, 2011, pp. 1--107. 


\bibitem[En89]{En89}
R.~Engelking, 
``General topology''. 
Second edition. Sigma Series in Pure Mathematics, 6. Heldermann Verlag, Berlin, 1989.

\bibitem[Gl15]{Gl15}
H.~Gl\"ockner, 
{\it Fundamentals of submersions and immersions between infinite-dimensional manifolds}, 
Preprint  arXiv:1502.05795v4 [math.DG].

\bibitem[HM92]{HM92}
R.~Harte, M.~Mbekhta, 
{\it On generalized inverses in $C^*$-algebras}. 
Studia Math. {\bf 103} (1992), no. 1, 71--77. 

\bibitem[Ho75]{Ho75}
K.H.~Hofmann, 
\textit{Th\'eorie directe des groupes de Lie}. II.  
S\'eminaire P. Dubreil (27e ann\'ee: 1973/74), Alg\`ebre, Fasc. 1, Exp. No. 2. 16 pp. Secr\'etariat Math\'ematique, Paris, 1975.

\bibitem[Ko01]{Ko01}
J.J.~Koliha, 
{\it Continuity and differentiability of the Moore-Penrose inverse in $C^*$-algebras}. 
Math. Scand. {\bf 88} (2001), no. 1, 154--160.

\bibitem[KM97]{KrMi}
A.~Kriegl, P.W.~Michor, 
``The convenient setting of global analysis''. 
Mathematical Surveys and Monographs, 53. American Mathematical Society, Providence, RI, 1997.

\bibitem[Lan01]{La}
S.~Lang, 
``Fundamentals of differential geometry''. 
Springer, 2001. 
	
\bibitem[Law98]{Lw98}
M.V.~Lawson, 
``Inverse semigroups''. 
World Scientific Publishing Co., Inc., River Edge, NJ, 1998.

\bibitem[LR12]{LeRo12}
J.~Leiterer, L.~Rodman, 
{\it Smoothness of generalized inverses}. 
Indag. Math. (N.S.)  {\bf 23} (2012), no. 3, 487--521. 

\bibitem[Mac87]{Ma}
K.~Mackenzie 
``Lie groupoids and Lie algebroids in differential geometry''. 
London Mathematical Society Lecture Note Series, 124. Cambridge University Press, Cambridge, 1987.

\bibitem[Mac05]{Ma05}
K.C.H.~Mackenzie, 
``General theory of Lie groupoids and Lie algebroids.'' 
London Mathematical Society Lecture Note Series, 213. Cambridge University Press, Cambridge, 2005.

\bibitem[MO92]{MaOu92}
J.~Margalef Roig, E.~Outerelo Dom\'\i nguez, 
``Differential topology.'' 
North-Holland Mathematics Studies, 173. 
North-Holland Publishing Co., Amsterdam, 1992. 


\bibitem[Mar08]{Mar}
E.~Mart\'\i nez 
{\it Variational calculus on Lie algebroids}. 
ESAIM Control Optim. Calc. Var. {\bf 14}  (2008), no. 2, 356--380.


\bibitem[MM03]{MoMr}
I.~Moerdijk, J.~Mrcun, 
``Introduction to foliations and Lie groupoids''. 
Cambridge University Press, 2003.

\bibitem[OJS15]{OJS15}
A.~Odzijewicz, G.~Jakimowicz, A.~Sli\.{z}ewska, 
{\it Banach-Lie algebroids associated to the groupoid of partially invertible elements of a $W^*$-algebra}. 
J. Geom. Phys. {\bf 95} (2015), 108--126.

\bibitem[OS16]{OS11}
A.~Odzijewicz, A.~Sli\.{z}ewska,  
{\it Groupoids and inverse semigroups associated to $W^*$-algebras}. 
J. Symplectic Geom. {\bf 14} (2016), no.~3, 687--736.  






\bibitem[Pe12]{Pe}
F.~Pelletier,
{\it Integrability of weak distributions on Banach manifolds}. 
Indag. Math. (N.S.) {\bf 23}  (2012), no. 3, 214--242. 


\bibitem[Ph87]{Ph}
J.~Philips,
{\it The holonomic imperative and the homotopy groupoid of a foliated manifold}. 
Rocky Mountain J. Math. {\bf 17}  (1987), no. 1, 151--165. 

\bibitem[PR87]{PoRe87}
H.~Porta, L.~Recht, 
{\it Minimality of geodesics in Grassmann manifolds}. 
Proc. Amer. Math. Soc.  {\bf 100}  (1987), no. 3, 464--466.

\bibitem[PS86]{PS86}
A.~Pressley, G.~Segal, 
``Loop groups.'' 
Oxford Mathematical Monographs. Oxford Science Publications.
The Clarendon Press, Oxford University Press, New York, 1986. 

\bibitem[SW15]{SW15}
A.~Schmeding, C.~Wockel, 
{\it The Lie group of bisections of a Lie groupoid}. 
Ann. Global Anal. Geom. {\bf 48} (2015), no. 1, 87--123. 

\bibitem[SW16]{SW16}
A.~Schmeding, C.~Wockel, 
{\it (Re)constructing Lie groupoids from their bisections and applications to prequantisation}. 
Differential Geom. Appl. {\bf 49} (2016), 227--276.


\end{thebibliography}
\end{document}